\documentclass[reqno,10pt]{amsart}
\usepackage{amssymb,amsthm,colordvi}

\usepackage{color,epsf,graphicx,graphics}

\usepackage{latexsym,amsfonts,amssymb,amsmath,amsthm,calc,color,colordvi,epsf,graphicx,graphics,float,upgreek,caption}

\usepackage[ansinew]{inputenc}

\numberwithin{equation}{section}

\newtheorem{defi}{Definition}[section]
\newtheorem{thm}[defi]{Theorem}

\newcommand{\bpr}{\begin{proof}[Proof]}  
\newcommand{\epr}{\end{proof}}
\newcommand{\beq}{\begin{equation}}
\newcommand{\eeq}{\end{equation}}
\newcommand{\bce}{\begin{center}}
\newcommand{\ece}{\end{center}}
\newcommand{\be}{\begin{enumerate}}  
\newcommand{\ee}{\end{enumerate}}
\newcommand{\bc}{\begin{center}}
\newcommand{\ec}{\end{center}}

\renewcommand{\rho}{\varrho}

\def\t{\tau}

\def\RR{\mathbb R}
\def\CC{\mathbb C}

\def\N{\mathbb N}

\def\cA{\mathcal A}
\def\cB{\mathcal B}

\def\cD{\mathcal D}
\def\cG{\mathcal G}

\def\cE{\mathcal E}

\def\cF{\mathcal F}

\def\cSM{\mathcal{SM}}
\def\cSX{\mathcal{SX}}

\def\bea{\begin{eqnarray}}
\def\eea{\end{eqnarray}}
\def\beas{\begin{eqnarray*}}
\def\eeas{\end{eqnarray*}}
\def\l{{\sf k}}

\def\nn{\nonumber}

\begin{document}

\title[Rayleigh-Taylor Instability for the Verigin Problem]
{ The Rayleigh-Taylor Instability for the Verigin Problem with and without Phase Transition}

\author{Jan Pr\"uss}
\address{Martin-Luther-Universit\"at Halle-Witten\-berg\\
         Institut f\"ur Mathematik \\
         D-06120 Halle, Germany}
\email{jan.pruess@mathematik.uni-halle.de}

\author{Gieri Simonett}
\address{Department of Mathematics\\
        Vanderbilt University\\
        Nashville, Tennessee\\
        USA}
\email{gieri.simonett@vanderbilt.edu}

\author{Mathias Wilke}
\address{Martin-Luther-Universit\"at Halle-Witten\-berg\\
         Institut f\"ur Mathematik \\
         D-06120 Halle, Germany}

\email{mathias.wilke@mathematik.uni-halle.de}

\thanks{This work was supported by a grant from the Simons Foundation (\#426729, Gieri Simonett).}


\begin{abstract}
Isothermal compressible two-phase flows in a capillary  are modeled with and without phase transition in the presence of gravity, employing Darcy's law for the velocity field. It is shown that the resulting systems are thermodynamically consistent in the sense that the available energy is a strict Lyapunov functional. In both cases, the equilibria with flat interface are identified.  It is shown that the problems are well-posed in an $L_p$-setting and generate local semiflows in the proper state manifolds. The main result concerns the stability of equilibria with flat interface, i.e. the Rayleigh-Taylor instability.
\end{abstract}

\maketitle

\section{Introduction}
{
The Verigin problem has been proposed in order to describe the process of pumping a viscous liquid into a porous medium 
which contains another fluid. This situation  occurs,  for instance,  when a porous medium containing oil is being flooded by water 
in order to displace the oil.
The resulting model is the compressible analogue to the Muskat problem in which the phases are considered incompressible. 

While there is an extensive amount of mathematical research on the Muskat problem, see for instance the extensive list of
references in \cite{EMW15, Ma17,PrSi16a}, there is only scarce work on the Verigin problem.
The papers \cite{BiSo00, Fro99, Fro03, Rad04, Tao97, TaYi96, Xu97} address local existence in some special cases, 
mostly excluding surface tension.
In \cite{PrSi16b},  the authors have developed a dynamical theory for the Verigin problem with and without phase transition. 
This includes local well-posedness, thermodynamical consistency,  identification of the equilibria, discussion of their stability, the local 
semiflows on the proper state manifolds, as well as convergence to equilibrium of solutions which do not develop singularities. 

In this manuscript, we consider the case where one fluid lies above the other one and gravity is acting on the fluids.
It is well-known  
that  the case where a fluid of higher density overlies a lighter one can lead to an instability, the 
famous Rayleigh-Taylor instability, see for instance \cite{Zh17a, Zh17b}.
This effect has been studied in the incompressible case for the Muskat problem with surface tension,
see for instance \cite{EEM13,EsMa11,EMM12a} and the references listed therein,
but also for the full (compressible or incompressible) Navier-Stokes equations \cite{GuTi10, JTW16, PrSi10, WaTi12, Wi13}.

It is the aim of this paper to study the Rayleigh-Taylor instability for the {\em Verigin problem}. }
We consider the problem in the setting of a capillary. To be more precise, let $G\subset \RR^{d-1}$, $d\ge 2$ be a bounded domain with 
$C^{4}$ (\"uberall ersetzt) -boundary $\partial G$, let $\overline{h},\underline{h}>0$, and define the finite capillary by means of $\Omega\times(-\underline{h},\overline{h})$. This set is decomposed into two parts, {\em the phases}, according to 
$$\Omega_1 =\{(x,y):\,x\in G,\,y\in (-\underline{h},h(x))\},\quad 
\Omega_2=\{(x,y):\,x\in G,\,y\in (h(x),\overline{h})\}.$$ 
Hence the interface is the graph $\Gamma =\{(x,h(x)):\,x\in G\}$, a free boundary. 
\begin{center}
\includegraphics[width=4.5cm]{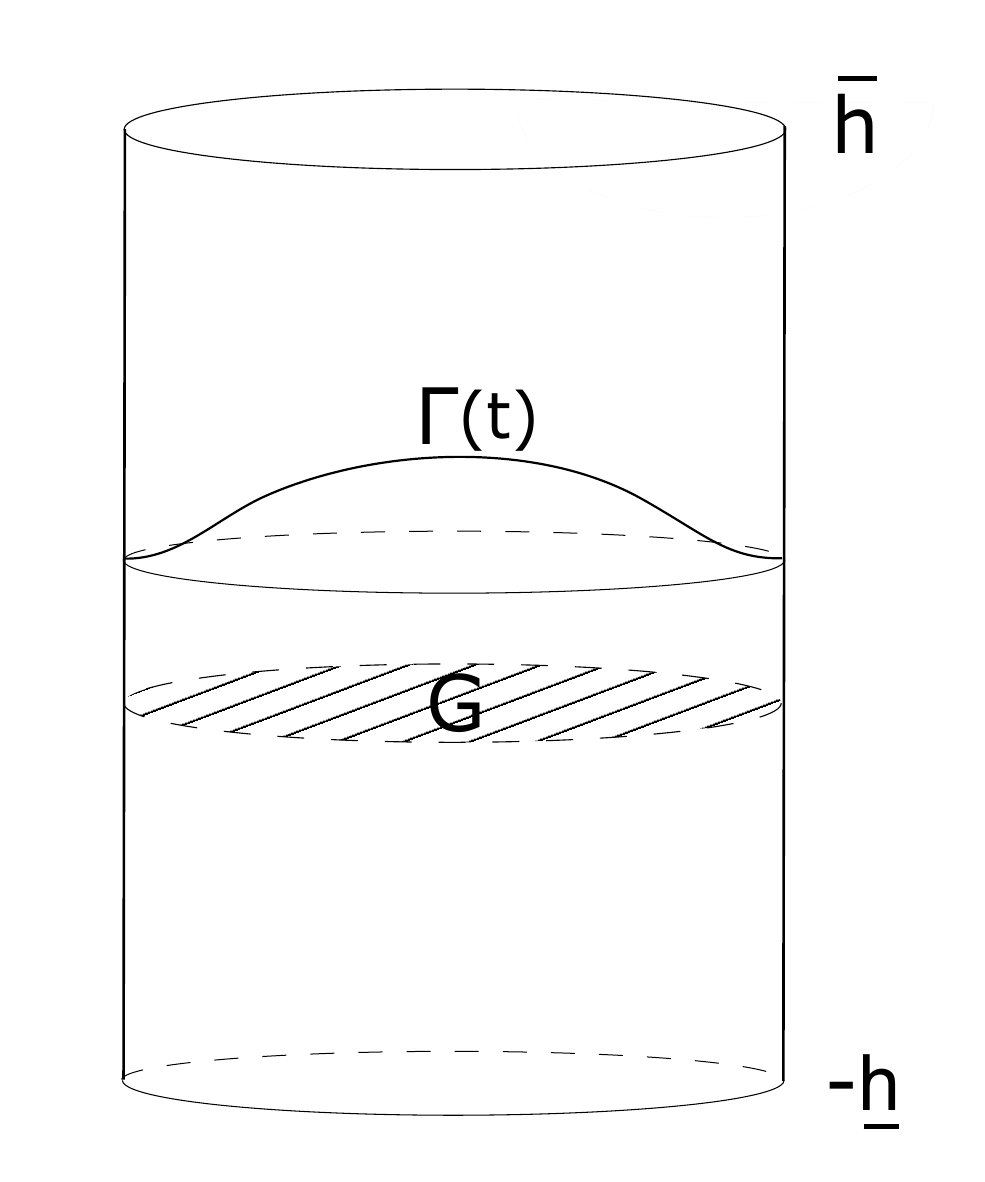}
\captionof{figure}{The capillary}
\end{center}
Let $u$ denote the velocity, $\varrho>0$ the density, and $p$ the pressure fields in $\Omega$. By $\nu =\nu_\Omega$ we designate the outer normal of $\Omega$, by $\nu_\Gamma$ the normal of $\Gamma$ pointing upwards, and by $\nu_G$ the outer normal of $G$, which is identified with its trivial extension to $\RR^d$. The jump of a quantity $w$ across $\Gamma$ is indicated by $[\![w]\!]=w_2-w_1$, and the unit vectors in $\RR^d$ are named ${\sf e}_j$, $j=1,\ldots,d$.

To state the model, recall that conservation of mass reads
\begin{equation}\label{mass}
\begin{aligned}
\partial_t\varrho + {\rm div}( \varrho u)&=0 &&\mbox{in } \Omega\setminus\Gamma,\\
\mbox{}[\![\varrho(u\cdot\nu_\Gamma-V_\Gamma)]\!] &=0  && \mbox{on } \Gamma,
\end{aligned}
\end{equation}
where $V_\Gamma$ denotes the normal velocity of the interface $\Gamma$ in the direction of $\nu_\Gamma$.
On the part $\partial\Omega\setminus\partial\Gamma$ of the outer boundary we require $u\cdot\nu_\Omega=0$, so that the flow induced by the velocity field $u$ does not leave $\Omega$.
The jump condition in \eqref{mass} shows that the {\em phase flux} 
$j_\Gamma:=\varrho(u\cdot\nu_\Gamma-V_\Gamma)$
is uniquely defined on $\Gamma$. We may then rewrite the jump condition as follows.
$$ V_\Gamma =[\![\varrho u\cdot\nu_\Gamma]\!]/[\![\varrho]\!],\quad j_\Gamma = [\![u\cdot\nu_\Gamma]\!]/[\![1/\varrho]\!].$$
On the interface we employ the {\em Laplace-Young law} and the $90$-degree angle condition
\begin{equation}\label{LYlaw}
\begin{aligned}
\mbox{}[\![p]\!]-\sigma H_\Gamma &=0  && \mbox{on } \Gamma,\\
\nu_G\cdot\nu_\Gamma &=0   && \mbox{on } \partial\Gamma,
\end{aligned}
\end{equation}
where $\sigma>0$ denotes the constant coefficient of surface tension and $H_\Gamma$ the $(d-1)$-fold mean curvature of $\Gamma$.

The total available energy of problem \eqref{mass} is given by
\begin{equation}
\label{available-energy}
 {\sf E}_a = \int_\Omega [\varrho\psi(\varrho) + \gamma \varrho y]\, d(x,y) + \sigma\! \int_\Gamma d\Gamma,
\end{equation}
i.e.\ the sum of the total free, potential, and surface energy. Note that in our context there is no kinetic energy.
Here $\psi$ means the mass-specific free energy density, which also depends on the phases,
and $\gamma$ is the acceleration of gravity.
In the next section we show
$$ \frac{d}{dt}{\sf E}_a = \int_\Omega u\cdot(\nabla p + \gamma \varrho\, {\sf e}_d)\,d(x,y) 
+\int_\Gamma[([\![p]\!]-\sigma H_\Gamma) V_\Gamma + [\![\varphi]\!]j_\Gamma ]\,d\Gamma,$$
where we already used {\em Maxwell's law} 
\begin{equation}
\label{Maxwell}
p(\varrho)=\varrho^2\psi^\prime(\varrho)
\end{equation}
and $\varphi(\varrho):= \psi(\varrho)+\varrho\psi^\prime(\varrho)$ as  abbreviations; observe the relation $p^\prime(\varrho)= \varrho\varphi^\prime(\varrho)$.

To obtain a closed model, we use Maxwell's law for the pressure, {\em Darcy's law} with gravity for the velocity, and a constitutive law for the phase flux $j_\Gamma$.
We ought to distinguish two cases.\vspace{3mm}\\
{\bf (i) No phase transition.}
Here we impose
\begin{align}\label{NoPT}
u=-k(\nabla p + \gamma \varrho\,{\sf e}_d),\quad p=\varrho^2 \psi^\prime(\varrho),\quad  j_\Gamma=0.
\end{align}
The constant $k>0$ is called {\em permeability} of the fluid; it may depend on the phase. 
Vanishing phase flux $j_\Gamma=0$ is equivalent to the relations
$$ [\![u\cdot\nu_\Gamma]\!]=0,\quad V_\Gamma=u\cdot\nu_\Gamma,$$
{hence the interface $\Gamma$ is solely advected with the fluid flow.}

\noindent
{\bf (ii) With phase transition.}
In this case, the constitutive laws read
\begin{align}\label{PT}
u=-k(\nabla p + \gamma \varrho \,{\sf e}_d),\quad p=\varrho^2 \psi^\prime(\varrho),\quad  [\![\varphi (\varrho)]\!]=0.
\end{align}
Note that the normal velocity $u\cdot\nu_\Gamma$ may jump across the interface $\Gamma$. The evolution of the interface $\Gamma$ is determined by the equation
$$[\![\rho]\!]V_\Gamma=[\![\rho u\cdot\nu_\Gamma]\!].$$
Here we assume $[\![\rho]\!]\neq 0$.

For the energy dissipation, these constitutive laws imply
$$ \frac{d}{dt}{\sf E}_a = - \frac{1}{k}\int_\Omega |u|^2\, d(x,y) = -k \int_\Omega|\nabla p +\gamma\varrho \,{\sf e}_d|^2\, d(x,y)\leq0.$$
Hence the total available energy is a Lyapunov functional for the problem. Even more, dissipation vanishes, i.e.\ $d/{\sf E}_a dt=0$, if and only if
\begin{equation}
\label{equilibria}
\begin{aligned}
\nabla p + \gamma\varrho(p){\sf e}_d &=0 &&\mbox{in }\; \Omega\setminus\Gamma,\\
[\![p]\!] -\sigma H_\Gamma &=0 &&\mbox{on }\; \Gamma,\\
\nu_G \cdot\nu_\Gamma &=0 && \mbox{on }\; \partial\Omega\cap\partial\Gamma,
\end{aligned}
\end{equation}
in case {\bf (i)}, and in addition $[\![\varphi]\!]=0$ on $\Gamma$ in case {\bf (ii)}.
This means that $(p,\Gamma)$ is an equilibrium, thereby showing that ${\sf E}_a$ is a strict Lyapunov functional. 
In particular, the problems are thermodynamically consistent.

Here we consider $\psi\in C^3(0,\infty)$ as given, assuming $\psi^\prime(s)>0$ as well as $\varphi^\prime(s)>0$ for all $s>0$. 
Then $p$, defined by Maxwell's law, is positive and strictly increasing, as $p^\prime(s)= s\varphi^\prime(s)>0$ by assumption. 
Therefore, we may invert Maxwell's law to obtain the {\em equation of state}
$\varrho = \varrho(p)$.

We will be interested in the stability properties of equilibria $(p_*,\Gamma_*)$ such that $\Gamma_*=\{(x,h_*):x\in G\}$ with a constant $h_*\in (\underline{h},\overline{h})$. Such equilibria will be called {\em flat}.
Then any interface $\Gamma(t)$ which is $C^2$-close to $\Gamma_*$ can be represented as a graph $\Gamma(t)=\{(x,h(t,x)):x\in G\}$, {see e.g.\ \cite{PrSi16}.} For such interfaces we have
$$ H_\Gamma = {\rm div}_x(\beta(h)\nabla_x h),\quad \beta(h)=(1+|\nabla_xh|^2)^{-1/2},$$
and
$$ V_\Gamma = \beta(h) \partial_t h,\quad \nu_\Gamma = \beta(h)\left[\begin{array}{c} -\nabla_x h\\ 1\end{array}\right].$$
Then \eqref{LYlaw} reads
\begin{equation}\label{LYlawh}
\begin{aligned}
{ }[\![p]\!] -\sigma {\rm div}_x(\beta(h)\nabla_x h)&=0 &&\text{on}\;\; \Gamma,\\
\partial_{\nu_G} h &=0 && \text{on}\;\; \partial G. 
\end{aligned}
\end{equation}
The jump condition in \eqref{mass} becomes in case {\bf (i)}
\begin{equation}
\begin{aligned} \label{dynhi}
{ }[\![k( \partial_y p -\nabla_x h\cdot\nabla_x p +\gamma\varrho(p))]\!]&=0 &&\text{on }\; \Gamma,\\
\partial_t h + k( \partial_y p -\nabla_x h\cdot\nabla_x p +\gamma\varrho(p))&=0  && \text{on }\;\Gamma,
\end{aligned}
\end{equation}
and in case {\bf (ii)} it reads
\begin{equation} \label{dynhii}
\begin{aligned}
{ }[\![\varphi(\varrho)]\!]&=0  && \text{on }\; \Gamma,\\
[\![\varrho(p)]\!]\partial_t h + [\![\varrho(p) k( \partial_y p -\nabla_x h\cdot\nabla_x p +\gamma\varrho(p))]\!]&=0  && \text{on }\; \Gamma.\\
\end{aligned}
\end{equation}
We shall show in Section 3 that problems \eqref{mass}, \eqref{LYlawh}, \eqref{dynhi} as well as \eqref{mass}, \eqref{LYlawh}, \eqref{dynhii}
are well-posed in an $L_p$-setting and generate local semiflows in their proper state manifolds $\cSM_i$ resp.\ $\cSM_{ii}$ 
to be be defined {at the end of} Section 3.

Our main results in this paper concern stability of flat equilibria. 
We prove in Section 4 that a flat equilibrium $(p_*,\Gamma_*)$ is stable in its state manifold in case {\bf (i)} if  
$\sigma\mu_1> \gamma[\![\varrho(p_*)]\!],$
and unstable if this inequality is reversed. 
 Here $\mu_1>0$ denotes the first nontrivial eigenvalue of the negative Neumann-Laplacian $-\Delta_N$ {in $L_2(G)$, $G\subset\RR^{d-1}$.} 
 If phase transition is present, i.e.\ in case {\bf(ii)}, we have stability if 
 $$\sigma\mu_1> \gamma[\![\varrho(p_*)]\!], \;\; \text{and in addition} \;\;
[\![\varrho(p_*)]\!](\varrho_1(p_*(\overline{h}))-\varrho_2(p_*(-\underline{h})))>0 ,$$ 
and instability if at least one of the inequalities is reversed.
 This is the Rayleigh-Taylor instability for the Verigin problem in a capillary with and without phase transition.
{Observe that if $[\![\rho(p_*)]\!]<0$, then one always has stability in case {\bf (i)} and also in case {\bf (ii)}, if in addition
$$\rho_1(p_*(\overline{h}))-\rho_2(p_*(\underline{h}))<0.$$}
{Setting $\sigma_*:=\gamma [\![ \varrho(p_*)]\!]/\mu_1$ we observe that 
$\sigma>\sigma_*$ yields stability (if additionally $[\![\varrho(p_*)]\!](\varrho_1(p_*(\overline{h}))-\varrho_2(p_*(-\underline{h})))>0$
in case (ii)). Hence, large surface tension has a stabilizing effect for the Verigin problem in a capillary.}

These results are based on a precise analysis of the spectra of the full linearizations of the problem at the given flat equilibrium $(p_*,\Gamma_*)$, and on the generalized principle of linearized stability \cite{PSZ09}. The proof follows the ideas in the monograph Pr\"uss and Simonett \cite{PrSi16}, Chapters 10 and 11, and in the papers Pr\"uss and Simonett \cite{PrSi10, PrSi16b} and in Wilke \cite{Wi13}.

\subsection{Notations} {Here and in the sequel, $W^s_p$ denote the Sobolev-Slobodeckii and $H^s_p$  the Bessel potential spaces,
respectively. We recall that $W^k_p=H^k_p$ for $k\in\N$ and $p\in (1,\infty)$. Let $X$ be an arbitrary Banach space.  For $T>0$ and $1<p<\infty$, let $L_{p,\mu}(0,T;X)$ denote the vector-valued weighted $L_p$-space
$$L_{p,\mu}(J;X) := \{u : (0,T)\to X:t\mapsto t^{1-\mu} u(t)\in L_p(0,T;X)\}$$
where $\mu\in (1/p, 1]$. The corresponding Sobolev space $H^1_{p,\mu}(0,T;X)$ is being defined as
$$H^1_{p,\mu}(0,T;X) := \{u\in L_{p,\mu}(J;X)\cap W_{1,loc}^1 (0,T;X) : \dot{u}\in L_{p,\mu}(J;X)\}$$
and a similar definition is used for higher order weighted Sobolev spaces. The weighted versions $W_{p,\mu}^s(0,T;X)$ and $H_{p,\mu}^s(0,T;X)$ of the Sobolev-Slobodeckii and the Bessel potential spaces are defined via real and complex interpolation of weighted Sobolev spaces, respectively. 

Finally, we define 
$${_0}W_{p,\mu}^s(0,T;X):=\{u\in W_{p,\mu}^s(0,T;X):u(0)=0\},$$
whenever the time trace $u(0)=\operatorname{tr}_{t=0} u$ exists.

Let $X,Y$ be Banach spaces. By $\cB(X,Y)$ we denote the set of all bounded and linear operators as mappings from $X$ to $Y$. The open ball in $X$ with center $x_0\in X$ and radius $r>0$ is denoted by $B_X(x_0,r)$.

}

\section{Available Energy and Equilibria}
\noindent
{\bf (a) Energy Dissipation.}
We consider the total available energy defined before in Section 1. 
Along a sufficiently smooth solution, the standard surface transport theorem yields
$ \frac{d}{dt}\int_\Gamma d\Gamma = -\int_\Gamma H_\Gamma V_\Gamma\, d\Gamma.$
On the other hand, the standard transport equation implies with conservation of mass \eqref{mass}
\begin{align*}
\frac{d}{dt}\int_\Omega \varrho y\, d(x,y) &= \int_\Omega \partial_t\varrho y\, d(x,y) - \int_\Gamma [\![\varrho]\!]V_\Gamma h\,d\Gamma\\
&= -\int_\Omega y\, {\rm div}(\varrho u)\, d(x,y) - \int_\Gamma [\![\varrho]\!] h\, d\Gamma\\
&= \int_\Omega \varrho u\cdot {\sf e}_d\, d(x,y) +  
\int_\Gamma \Big([\![\varrho u\cdot\nu_\Gamma]\!]-[\![\varrho V_\Gamma]\!]\Big)h\, d\Gamma\\
&=\int_\Omega \varrho u\cdot {\sf e}_d\,d(x,y).
\end{align*}
In a similar way, employing also Maxwell's law, we obtain
\begin{align*}
\frac{d}{dt}&\int_\Omega \varrho\psi(\varrho)\, d(x,y)\\ 
&= \int_\Omega (\varrho\psi(\varrho))^\prime \partial_t \varrho \, d(x,y) 
- \int_\Gamma[\![\varrho\psi(\varrho)]\!] V_\Gamma\, d\Gamma\\
&= -\int_\Omega \Big( (\varrho\psi(\varrho))^\prime \varrho\, {\rm div}\, u + u\cdot \nabla(\varrho\psi(\varrho)) \Big)\,d(x,y) -
\int_\Gamma[\![\varrho\psi(\varrho)V_\Gamma]\!]\, d\Gamma\\
&= -\int_\Omega\Big( (\varrho\psi(\varrho))^\prime \varrho - \varrho\psi(\varrho)\Big){\rm div}\, u\, d(x,y) 
+\int_\Gamma[\![\varrho(u\cdot\nu_\Gamma-V_\Gamma)\psi(\varrho]\!]\,d\Gamma\\
&= -\int_\Omega p(\varrho)\,{\rm div}\, u\, d(x,y) + \int_\Gamma [\![\psi(\varrho)]\!]j_\Gamma\, d\Gamma\\
&= \int_\Omega u\cdot\nabla p(\varrho)\,d(x,y) 
	+\int_\Gamma \Big([\![p(\varrho)]\!] V_\Gamma+[\![\psi(\varrho)+p(\varrho)/\varrho]\!]j_\Gamma \Big)\,d\Gamma.
\end{align*}
In summary, with the definition of $\varphi$, these identities yield
$$ \frac{d}{dt}{\sf E}_a = \int_\Omega u\cdot(\nabla p(\varrho) + \gamma \varrho\, {\sf e}_d)\,d(x,y) 
+\int_\Gamma[([\![p(\varrho)]\!]-\sigma H_\Gamma) V_\Gamma + [\![\varphi(\varrho)]\!]j_\Gamma ]\,d\Gamma.$$
Hence by the constitutive laws we get
$$ \frac{d}{dt}{\sf E}_a = - \frac{1}{k}\int_\Omega |u|^2\,d(x,y) = -k \int_\Omega|\nabla p +\gamma\varrho {\sf e}_d|^2\,d(x,y)\leq0,$$
as asserted in Section 1.

\bigskip

\noindent
{\bf (b) Equilibria.}
Next we discuss the set $\cE$ of flat equilibria of the problems, where we again have to distinguish two cases.\vspace{2mm}\\
{\bf (i)} {\em No phase transition.}
In this case we employ the equation of state $\varrho=\varrho(p)$ and define another function $\phi(p)$ by means of $\phi^\prime(p) =1/\varrho(p)$
and, say, $\phi(1)=0$. Then we obtain
$$\nabla \phi(p) = \phi^\prime(p)\nabla p = -\gamma {\sf e_d},$$
which yields
\begin{equation}
\label{eqp}
p(y) = \left\{
\begin{aligned} 
&\phi_1^{-1}(a_1 -\gamma y), && y\in [-\underline{h},h],\\ 
&\phi^{-1}_2(a_2-\gamma y),  && y\in[h,\overline{h}],
\end{aligned}
\right.
\end{equation}
where $h\in (-\underline{h},\overline{h})$ is fixed (defining the flat interface),
and $a_j$ are arbitrary constants.
The function $p$ is well well-defined, provided
$$ a_1-\gamma[-\underline{h},h] \subset im(\phi_1)\quad\mbox{and}\quad a_2-\gamma[h,\overline{h}] \subset im(\phi_2).$$
The Laplace-Young law for the flat interface
becomes $[\![p]\!]=0$, which yields the equation
$$ f(a_1,a_2,h):= \phi^{-1}_2(a_2-\gamma h)-  \phi_1^{-1}(a_1 -\gamma h)=0.$$
This equation may or may not have solutions $(a_1,a_2,h)\in \RR^2\times(-\underline{h},\overline{h})$. 
Any solution defines a flat equilibrium if $p$ is well-defined. 
With $[\phi^{-1}_j]^\prime = 1/\phi^\prime_j\circ \phi_j^{-1} = \varrho_j\phi_j^{-1}$ we have
$$ f^\prime(a_1,a_2,h)=( -\varrho_1,\varrho_2, -\gamma[\![\varrho]\!])\neq 0,$$
where $\varrho_j=\varrho_j(p(h))$ and  $[\![\varrho]\!]=\varrho_2(p(h))-\varrho_1(p(h)).$
Hence, the set of zeros of $f$ is a manifold of dimension 2 of class $C^3$ if it is nonempty, and so the set of flat equilibria is a manifold of dimension 2.

Next let us exploit conservation of the masses of the phases, i.e.
$$ {\sf M}_j(\varrho_j,\Gamma):=\int_{\Omega_j} (\varrho_j\circ p) \, d(x,y) = \int_{\Omega_{0j}} \varrho_j(p_0)\,d(x,y):={\sf M_{0j}},
\quad j=1,2.$$
Defining functions
\begin{equation*}
\begin{aligned}
& M_1(a_1,a_2,h):= |G| \int_{-\underline{h}}^h \varrho_1(\phi_1^{-1}(a_1-\gamma y))\, dy, \\
& M_2(a_1,a_2,h):= |G| \int_h^{\overline{h}}\varrho_2(\phi_2^{-1}(a_2-\gamma y))\, dy,
\end{aligned}
\end{equation*}
and setting $g(a_1,a_2,h) = [M_1-{\sf M}_{01},M_2-{\sf M}_{02},f]^{\sf T}$, the zeros of $g$ correspond to the flat equilibria with prescribed masses of the phases.
For the derivative of $g$ we obtain
$$g^\prime(a_1,a_2,h) = \left[ \begin{array}{ccc} |G|c_1&0& |G|\varrho_1\\ 0& |G|c_2& -|G|\varrho_2\\ -\varrho_1&\varrho_2& -\gamma[\![\varrho]\!]\end{array}\right],$$
where $\varrho_j = \varrho_j(p(h))$ and  
\begin{equation*}
c_1 =  \int_{-\underline{h}}^h  (\varrho_1^\prime\circ p)(\varrho_1\circ p) \, dy, \quad
c_2=\int_h^{\overline{h}} (\varrho_2^\prime\circ p)(\varrho_2\circ p) \,dy.
\end{equation*}
For the determinant  we get
$$ {\rm det}\,g^\prime(a_1,a_2,h) = |G|^2 c_1c_2\left(  \frac{\varrho_1^2}{c_1}+\frac{\varrho_2^2}{c_2}-\gamma[\![\varrho]\!]\right).$$
Hence a flat equilibrium with prescribed masses of the phases
is isolated if
$$  \frac{\varrho_1^2}{c_1}+\frac{\varrho_2^2}{c_2}\neq \gamma[\![\varrho]\!] .$$
Observing that with $p^\prime = -\gamma (\varrho\circ p)$ we have
$$ \gamma c_1= \varrho_1(p(-\underline{h})) -  \varrho_1(p(h)),\quad \gamma c_2 = \varrho_2(p(h))-\varrho_2(p(\overline{h})),$$
and some simple algebra shows
\begin{equation}\label{lcuni}
\gamma [\![\varrho]\!] <  \frac{\varrho_1^2}{c_1}+\frac{\varrho_2^2}{c_2}.
\end{equation}
Therefore, all flat equilibria are isolated when the masses are prescribed. 
This inequality will also play an important role for the stability of flat equilibria.

\bigskip

\noindent
{\bf (ii)} {\em With phase transition.}
In this case it is more convenient to work with the function $\varphi$. We have
$$\nabla (\varphi\circ\varrho) = \varphi^\prime(\varrho) \nabla \varrho 
= p^\prime(\varrho)\nabla\varrho/\varrho = \nabla p(\varrho)/\varrho=-\gamma{\sf e_d},$$
hence we obtain
\begin{equation}
\label{eqpii}
\varrho(y) = \left\{
\begin{aligned} 
&\varphi^{-1}_1(a_1 -\gamma y),  && y\in [-\underline{h},h],\\ 
&\varphi^{-1}_2(a_2-\gamma y),   && y\in[h,\overline{h}],
\end{aligned}\right.
\end{equation}
where $h\in (-\underline{h},\overline{h})$ is fixed (defining the flat interface), and $a_j$ are arbitrary constants.
The function $p$ is well well-defined, provided
$$ a_1-\gamma[-\underline{h},h] \subset im(\varphi_1)\quad\mbox{and}\quad a_2-\gamma[h,\overline{h}] \subset im(\varphi_2).$$
The jump condition $[\![\varphi(\varrho)]\!]=0$ on the interface implies $a_1=a_2=:a$, and the pressure jump condition $[\![p]\!]=0$,
 valid for a flat interface, yields
$$ f(a,h):= \varrho^2_2(h)\psi_2^\prime(\varrho_2(h))-\varrho^2_1(h)\psi_1^\prime(\varrho_1(h)) =0.$$
This equation may or may not have solutions, but any solution defines a flat interface, if $\varrho$ is well-defined on $[-\underline{h},\overline{h}]$. For the derivative of $f$ we obtain
$$ f^\prime(a,h)= ([\![\varrho]\!], -\gamma[\![\varrho]\!]),  $$
which is nontrivial if $[\![\varrho]\!]\neq0$; we call such flat equilibria {\em non-degenerate}. 
Therefore, the non-degenerate flat equilibria form a 1-dimensional manifold {of class $C^2$.}

Let us consider conservation of total mass
$${\sf M}(\varrho,\Gamma):=\int_\Omega \varrho \,d(x,y) = \int_\Omega \varrho(p_0)\,dx
=:{\sf M}_0.$$
This yields at a flat equilibrium
$$M(a,h) := |G|(\int_{-\underline{h}}^h \varphi^{-1}_1(a-\gamma y)\,dy + \int_h^{\overline{h}} \varphi^{-1}_2(a-\gamma y)\,dy).$$
For the derivative of $M$ we obtain by an easy computation
$$M^\prime(a,h) = (|G|c,-|G|[\![\varrho]\!]), \quad c= \int_{-\underline{h}}^h \frac{\varrho_1}{p_1^\prime(\varrho_1)}\,dy
+ \int_h^{\overline{h}}\frac{\varrho_2}{p_2^\prime(\varrho_2)}\,dy.$$
Therefore, as
$${\rm det}\left[\begin{array}{c} M^\prime(a,h)\\ f^\prime(a,h)\end{array}\right]= |G|[\![\varrho]\!]([\![\varrho]\!]-\gamma c),$$
we see that a non-degenerate flat equilibrium is isolated within the class of given mass ${\sf M}_0$, 
provided $[\![\varrho]\!]\neq \gamma c$.
As  
{$\varrho_j^\prime =-\gamma \varrho_j/p_j'(\varrho_j)$ }
we have
$$ \gamma c = [\varrho_1(-\underline{h})-\varrho_1(h)] + [ \varrho_2(h)-\varrho_2(\overline{h})],$$
hence
\begin{equation}\label{lcunii}
\gamma c -[\![\varrho]\!] = \varrho_1(-\underline{h})-\varrho_2(\overline{h}).
\end{equation}
This quantity will be important in the stability analysis in Section 4.

\bigskip

\noindent
{\bf (c)} {\em Critical Points of Available Energy}\\
Now we want to give a motivation of the stability conditions involved below by a variational analysis. 
The arguments here are not completely rigorous,
but are such that they indicate the physics behind the models.

Let us consider the critical points of the available energy ${\sf E}$ with the mass constraints ${\sf M}_j={\sf M}_{j0}$ in case {\bf (i)} and
${\sf M}={\sf M}_1+{\sf M}_2 = {\sf M}_0$ in case {\bf (ii)}.
At smooth functions $\varrho$ and interfaces $\Gamma$ we obtain for the first variations of ${\sf E}, {\sf M_j}$ and ${\sf M}$
in the direction of $(\tau, g)$,
where $\Gamma$ is varied in the direction of the normal vector field $g \nu_{\Gamma}$, with $g: \Gamma \to\RR$ a sufficiently smooth function
\begin{align*}
\langle {\sf E}^\prime(\varrho,\Gamma) |(\tau,g)\rangle &= \int_\Omega (\varphi(\varrho) +\gamma y)\tau \,d(x,y)
	-\int_\Gamma\big([\![\varrho\psi(\varrho) +\gamma\varrho y]\!] + \sigma H_\Gamma\big) g\, d\Gamma\\
\langle {\sf M}_1^\prime(\varrho,\Gamma)|(\tau,g)\rangle &= \int_\Omega \chi_1\tau \,d(x,y) +\int_\Gamma\varrho_1 g\, d\Gamma\\
\langle {\sf M}_2^\prime(\varrho,\Gamma)|(\tau,g)\rangle &= \int_\Omega \chi_2\tau \,d(x,y) -\int_\Gamma\varrho_2 g\, d\Gamma\\
\langle {\sf M}^\prime (\varrho,\Gamma)|(\tau,g)\rangle &= \int_\Omega \tau\,d(x,y) -\int_\Gamma[\![\varrho]\!] g\,d\Gamma.
\end{align*}
We recall that $\varphi= (\varrho\psi)^\prime = \psi +\varrho \psi^\prime$. 
Moreover, $\chi_j$ denotes the characteristic function of $\Omega_j,$ $j=1,2,$ respectively. 
The method of Lagrange multipliers at a critical point yields constants $\mu_j$ such that
$$ {\sf E}^\prime +\mu_1 {\sf M}^\prime_1 + \mu_2 {\sf M}_2^\prime =0,$$
where $\mu_1=\mu_2=:\mu$ in case {\bf (ii)}. Varying first $\tau$ and then $g$ we obtain
$$ \varphi_j(\varrho) + \gamma y + \mu_j =0 \;\mbox{ in }\; \Omega_j,$$
and
$$[\![ \varrho\psi(\varrho) + \gamma \varrho y + \varrho\mu]\!] +\sigma H_\Gamma=0.$$
From these relations we deduce in both cases that a critical point is an equilibrium, besides the angle condition.

Next, let a flat equilibrium $(p_*,\Gamma_*)$ be given, and consider the second variation of  ${\sf E}$ with mass constraints. If the critical point $(p_*,\Gamma_*)$ is  a local minimum of ${\sf E}$ with constraints ${\sf M}_j={\sf M}_{j0}$ resp.\ ${\sf M}={\sf M}_0$, then the second variation
$$ {\sf C}_i := {\sf E}^{\prime\prime} +\mu_1 {\sf M}_1^{\prime\prime}+\mu_2 {\sf M}_2^{\prime\prime},\quad \mbox{resp.} \quad
{\sf C}_{ii} := {\sf E}^{\prime\prime} +\mu {\sf M}^{\prime\prime}$$
must be positive semi-definite on ${\sf N}({\sf M}_1^\prime)\cap{\sf N}({\sf M}^\prime_2)$ resp.\ on ${\sf N}({\sf M}^\prime)$. It is not difficult to compute
${\sf C}_\l$ at the equilibrium to the result
$$ \langle{\sf C}_\l (\tau,g)|(\tau,g)\rangle = \int_\Omega \varphi_*^\prime \tau^2\, d(x,y) 
	- \int_{\Gamma_*} \big(\sigma H_{\Gamma_*}^\prime g \cdot g + \gamma[\![\varrho_*]\!] g^2\big)\,d\Gamma_*$$
in both cases $\l\in \{i,ii\}$. For this we have taken into account the relations for critical points from above.  
Here we have used the notation
$\varrho_*=\varrho(p_*)$, and $\varphi^\prime_*=\varphi^\prime(\varrho_*)$.
As $\varphi^\prime= p^\prime/\varrho 
= 1/\varrho^\prime\varrho$,
we obtain by the Cauchy-Schwarz inequality
$$ \big(\int_{\Omega_j} \tau \,d(x,y)\big)^2 
= \big(\int_{\Omega_j} \sqrt{\varrho_*^\prime\varrho_*}\cdot \sqrt{\varphi_*^\prime}\,\tau d(x,y)\big)^2
\leq (\int_{\Omega_j} \varrho_*^\prime\varrho_* \,d(x,y)) \int_{\Omega_j} \varphi_*^\prime \tau^2 \,d(x,y),$$
with equality if $ \tau= \alpha_j \varrho_*\varrho_*^\prime$ for some constant $\alpha_j$.
Here  we set $\varrho^\prime_*=\varrho^\prime(p_*)$.
As $(\tau,g)\in {\sf N}({\sf M}_1^\prime)\cap{\sf N}({\sf M}^\prime_2)$, i.e.
$$\int_\Omega \chi_{*j}\tau \,d(x,y) = (-1)^j\int_{\Gamma_*}\varrho_{*j}\, g \,d\Gamma_*,\quad j=1,2,$$
and setting
$$ \tau=\alpha_j \varrho_*^\prime\varrho_*,\quad \alpha_j =(-1)^j\int_{\Gamma_*}\varrho_{*j}\, g \,d\Gamma_*/d_j,\quad d_j
= \int_\Omega \chi_{*j} \varrho_*^\prime\varrho_* \,d(x,y)$$
the property that ${\sf C}_i$ is positive semi-definite implies
$$-\int_{\Gamma_*} (\sigma H^\prime_{\Gamma_*} g\cdot g + [\![\gamma\varrho_*]\!] g^2) \,d\Gamma_* 
+ d_1^{-1}(\int_{\Gamma_*}\varrho_{1*}\,g \,d\Gamma_*)^2
+  d_2^{-1}(\int_{\Gamma_*}\varrho_{2*}\,g \,d\Gamma_*)^2\ge 0, $$
for all $g\in H^2_2(\Gamma_*)$ with $\partial_\nu g=0$ on $\partial G$ by the angle condition. In a similar way, in case {\bf (ii)} we get the necessary condition
 $$-\int_{\Gamma_*} (\sigma H^\prime_{\Gamma_*} g\cdot g + [\![\gamma\varrho_*]\!] g^2) \,d\Gamma_* 
 + d^{-1}\left(\int_{\Gamma_*}[\![\varrho_{*}]\!] g \,d\Gamma_*\right)^2\ge 0,$$
 where $d=\int_\Omega \varrho^\prime_*\varrho_*\,d(x,y)$.

Now we use the fact that $(p_*,\Gamma_*)$ is flat, i.e.\ $\Omega_1=\{(x,y):\, x\in G,\, -\underline{h} < y < h_*\}$ and
$\Omega_2=\{(x,y):\, x\in G,\, h_* < y < \overline{h}\}$. Then $H_{\Gamma_*}^\prime = \Delta_x$, and employing the decomposition 
$g= g_0 +{\sf g}$ with  $\int_G g_0\, dx =0$, these conditions become
$$ \int_G \big(-\sigma \Delta_x g_0 -\gamma[\![\varrho_*]\!] g_0\big)g_0 \,dx + 
\left( \frac{\varrho_{1*}^2}{c_1}+ \frac{\varrho_{2*}^2}{c_2}- \gamma[\![\varrho_*]\!]\right)|G|{\sf g}^2\ge 0,$$
in case {\bf (i)}, and
$$  \int_G \big(-\sigma \Delta_x g_0 -\gamma[\![\varrho_*]\!] g_0\big)g_0 \,dx 
+ \big([\![\varrho_*]\!]^2/c- \gamma[\![\varrho_*]\!]\big)|G|{\sf g}^2\ge 0$$
in case {\bf (ii)}.
As $g_0$ and ${\sf g}$ are independent, setting ${\sf g}=0,$ these inequalities imply that the operator 
$- (\sigma \Delta_x  + \gamma[\![\varrho_*]\!] )$ is positive semi-definite, which means
$\mu_1\geq \mu_*=\gamma[\![\varrho_*]\!]/\sigma$, where $\mu_1>0$ denotes the smallest 
nontrivial eigenvalue of the negative Neumann-Laplacian on $G$.
On the other hand, setting $g_0=0$ we obtain the necessary  conditions
$$\frac{\varrho_{1*}^2}{c_1}+ \frac{\varrho_{2*}^2}{c_2}\geq \gamma[\![\varrho_*]\!]$$
in case {\bf (i)}, which always holds as we have seen above, and
$$ [\![\varrho_*]\!]^2\geq \gamma c[\![\varrho_*]\!]$$
in case {\bf (ii)}.

As a summary of the above  considerations we have the following result.
\begin{thm}\label{variation}
{\bf (a)} Suppose that $(p_*,\Gamma_*)$ is a critical point of the available energy with prescribed masses.
Then $(p_*,\Gamma_*)$ is an equilibrium. 
\vspace{1mm}\\
{\bf (b)} Let $(p_*,\Gamma_*)$ be a flat equilibrium, which is a local minimum of the available energy with mass constraints. 
Then
\vspace{1mm}\\
{\bf ($\alpha$)} $\sigma\Delta_N + \gamma[\![\varrho_*]\!]$ is negative semi-definite on $L_2$-functions on $\Gamma_*$ with mean zero.
\vspace{1mm} \\
{\bf ($\beta$)} If phase transition is present, then $\gamma c[\![\varrho_*]\!]\leq [\![\varrho_*]\!]^2$.
\end{thm}

\noindent
This result shows that the Verigin problems with and without phase transition are thermodynamically consistent and stable. Below we investigate the dynamic stability of flat equilibria in a rigorous way.
\goodbreak
\section{Local Well-Posedness}
The two problems in question read\vspace{2mm}\\
{\bf (i)} {\em Without phase transition}
\begin{equation} 
\begin{aligned}\label{p1}
\varrho^\prime(p)\partial_t p - {\rm div} ( \varrho(p)k(\nabla p +\gamma\varrho(p){\sf e}_d))&=0 &&\mbox{in }\; \Omega\setminus\Gamma,\\
\partial_\nu p +\gamma\varrho(p)\nu\cdot {\sf e}_d &=0 && \mbox{on }\; \partial\Omega\setminus\partial\Gamma,\\
[\![p]\!]-\sigma{\rm div}_x(\beta(h)\nabla_x h) &=0 && \mbox{on }\; \Gamma,\\
\partial_{\nu_G} h &=0 && \mbox{on }\; \partial G,\\
[\![ k(\partial_y p -\nabla_x h\cdot\nabla_x p +\gamma\varrho(p))]\!]&=0 && \mbox{on }\; \Gamma,\\
\partial_t h + k(\partial_y p -\nabla_x h\cdot\nabla_x p +\gamma\varrho(p))&=0 &&  \mbox{on }\; \Gamma,\\
p(0)=p_0 \mbox{ in } \Omega\setminus\Gamma, \quad h(0)&=h_0 && \mbox{in }\; G.
\end{aligned}
\end{equation}
{\bf (ii)} {\em With phase transition}
\begin{equation}
\begin{aligned} \label{p2}
\varrho^\prime(p)\partial_t p - {\rm div} ( \varrho(p)k(\nabla p +\gamma\varrho(p){\sf e}_d))&=0 &&\mbox{in } \;\Omega\setminus\Gamma,\\
\partial_\nu p +\gamma\varrho(p)\nu\cdot {\sf e}_d &=0 && \mbox{on }\; \partial\Omega\setminus\partial\Gamma,\\
[\![p]\!]-\sigma{\rm div}_x(\beta(h)\nabla_x h) &=0 && \mbox{on }\; \Gamma,\\
\partial_{\nu_G} h &=0 && \mbox{on }\;\partial G,\\
[\![\varphi(\varrho(p)))]\!]&=0 && \mbox{on }\; \Gamma,\\
[\![\varrho(p)]\!]\partial_t h + [\![\varrho(p)k(\partial_y p -\nabla_x h\cdot\nabla_x p +\gamma\varrho(p))]\!]&=0 && \mbox{on }\; \Gamma,\\
p(0)=p_0 \mbox{ in } \Omega\setminus\Gamma, \quad h(0)&=h_0 && \mbox{in }\; G.
\end{aligned}
\end{equation}
We recall that the interface $\Gamma=\Gamma(t)=\Gamma_{h(t)}$ is given as the graph
$$ \Gamma_{h(t)}=\{ (x,h(t,x)):\, x\in G\},$$
defined by the height function $h(t,x)$.

\subsection{Transformation to a Fixed Domain}
We want show local well-posedness for initial interfaces which are $C^2$-close to a constant one, which by shifting $y$ can be assumed w.l.o.g.\ to be
$\Sigma:=G\times\{0\}$. For this purpose, we transform the time-varying domains $\Omega\setminus \Gamma(t)$ to the fixed one $\Omega\setminus \Sigma$, by means of the variable transformation
$$ \theta(x,y)=\theta_{h(t)}(x,y) = (x,y+\chi(y) h(t,x)),\quad x\in \overline{G},\; y\in[-\underline{h},\overline{h}].$$
Here $\chi(y)\in[0,1]$ means a $C^\infty$-cut-off function which equals $1$ for {$-\underline{h}/3< y<\overline{h}/3$} and $0$ outside of $(-2\underline{h}/3,2\overline{h}/3)$. Then
$$\vartheta(x)=\vartheta_{h(t)}(x) = \theta_{h(t)}(x,0)= (x,h(t,x)),\quad x\in G,$$
parameterizes the unknown interface $\Gamma(t)$ as a graph over $G$.

With
$$D\theta(x,y) = \left[ \begin{array}{cc}I&0\\ \chi(y)\partial_xh(t,x)& 1+\chi^\prime(y)h(t,x)\end{array}\right]$$
it is clear that $\theta$ is a $C^2$-diffeomorphism of $\Omega\setminus\Sigma$ onto $\Omega\setminus\Gamma$, leaving 
 $\partial\Omega$ invariant,
provided $h\in C^2$ and $|h|_\infty < 1/2|\chi^\prime|_\infty$. Note that the derivatives of $\theta$ are bounded up to order $2$.

Next we take a function $\pi(y)$ to scale the pressure vertically, i.e.\ we will choose {$\pi(y)=1$} for well-posedness, 
but $\pi(y)= p_*(y)$ where $(p_*,\Sigma)$ is a flat equilibrium, in the case of stability considerations. We will assume accordingly $\pi\in BU\!C^2([-\underline{h},0)\cup(0,\overline{h}]) $ continuous at $y=0$.

With this scaling we define the new variable $v(t,x,y)$ by means of
$$ v(t,x,y) = \frac{p(t,\theta_{h(t)}(x,y))}{\pi(\theta_{h(t)}(x,y))}= \frac{p(t,x,y+ \chi(y)h(t,x))}{\pi(y+\chi(y)h(t,x))},$$
for $t\geq0$, $x\in G$, $y\in [-\underline{h},\overline{h}]$. Then $v$ lives on the fixed domain $\Omega\setminus\Sigma$.
We have
\begin{align*}
\partial_y(p\circ\theta) &= \partial_y((\pi\circ\theta)v) = (\pi\circ\theta) \partial_y v + v (\pi^\prime\circ\theta)(1+\chi^\prime h)\\
                  &= (1+\chi^\prime h) \partial_yp \circ\theta,
\end{align*}
hence
$$ \partial_y p \circ \theta = \frac{\pi\circ\theta}{1+\chi^\prime h} \partial_y v + (\pi^\prime\circ\theta)v.$$
Next, we obtain for the time derivative
\begin{align*}
\partial_t(\varrho(p\circ\theta)) &= \varrho^\prime(p\circ\theta) \partial_t (p\circ\theta) = \varrho^\prime(p\circ\theta)\big(\partial_tp\circ\theta + \chi\partial_th \partial_y p\circ\theta\big)\\
&=\partial_t \varrho(v\pi\circ\theta) = \varrho^\prime(v\pi\circ\theta) \big((\pi\circ\theta)\partial_tv + v (\pi^\prime\circ\theta) \chi\partial_th\big),
\end{align*}
and so
$$ (\varrho^\prime (p)\partial_t p)\circ\theta = \varrho^\prime(v\pi\circ\theta)(\pi\circ\theta)
 \Big(\partial_t v -\frac{\chi\partial_t h}{1+\chi^\prime h} \partial_y v \Big).$$
Finally,
\begin{align*}
\nabla_x(p\circ\theta)&= \nabla_x p\circ\theta + (\partial_yp\circ\theta ) \chi \nabla_x h\\
&= \nabla_x ( v\pi\circ\theta) = (\pi\circ\theta) \nabla_x v + v (\pi^\prime\circ\theta) \chi\nabla_x h,
\end{align*}
which implies
$$\nabla_xp\circ \theta = (\pi\circ\theta)\Big( \nabla_x v - \frac{\chi\nabla_x h}{1+\chi^\prime h}\partial_y v\Big).$$
For the transformation of the boundary and interface conditions, observe that $\chi=1$ and $\chi^\prime=0$ near $y=h$, provided $|h|_\infty <\min\{ \underline{h},\overline{h}\}/3$, as well as $\chi=\chi^\prime=0$ near $y=-\underline{h},\overline{h}$. Then the boundary conditions at the top and at the bottom of $\Omega$ become
$$ \pi \partial_y v + \pi^\prime v +\gamma\varrho(v\pi)=0,\quad   x\in G, \;  y\in\{-\underline{h}, \overline{h}\}.$$
In virtue of $\partial_{\nu_G} h=0$, at the lateral boundary of $\Omega$, we have
\begin{align*} 0&= (\nu_G\cdot \nabla_x p)\circ\theta = (\pi\circ\theta) \Big(\nu_G\cdot\nabla_x v 
- \frac{\chi \nu_G\cdot\nabla_x h}{1+\chi^\prime h} \partial_y v\Big)= (\pi\circ\theta) \partial_{\nu_G} v,
\end{align*}
which means
$$ \partial_{\nu_G} v =0,\quad x\in\partial G, \; y\in (-\underline{h},\overline{h})\setminus\{0\}.$$
Here we used the fact that the normal of $G$ is preserved by the transformation, i.e.\ $\nu_G(x,y)$ is transformed to $\nu_G(x,y+\chi(y) h)$.

On the interface $\Sigma=G\times\{0\}$, the Laplace-Young law and the angle condition become
\begin{align*}
[\![(\pi\circ h) v]\!] -\sigma {\rm div}_x(\beta(h) \nabla_x h)&=0 \quad \mbox{ on  } \Sigma,\\
\partial_{\nu_G} h &=0\quad \mbox{ on } \partial G.
\end{align*}
In case {\bf (i)} we further have on $\Sigma$
{$$ [\![ k (\pi\circ h) \left((1+|\nabla_x h|^2)\partial_y v -\nabla_x h\cdot\nabla_x v\right)  +k\left(\gamma \varrho((\pi\circ h)v) + (\pi^\prime\circ h)v\right)]\!]=0,$$}
and
{$$\partial_t h + k(\pi\circ h)\left((1+|\nabla_x h|^2)\partial_y v -\nabla_x h\cdot\nabla_x v\right)  +k\left(\gamma \varrho((\pi\circ h)v) + (\pi^\prime\circ h)v\right)=0.$$}
In case {\bf (ii)} these two conditions have, instead, to be replaced by
$$ [\![\varphi( \varrho((\pi\circ h)v))]\!]=0,\quad x\in G,$$
and
\begin{equation*}
\begin{split}
 [\![\varrho((\pi\circ h)v)]\!] \partial_t h + &[\![\varrho((\pi\circ h)v)k(\pi\circ h)\left((1+|\nabla_x h|^2)\partial_y v -\nabla_x h\cdot\nabla_x v\right)  \\
&\quad  +k\left(\gamma \varrho((\pi\circ h)v) + (\pi^\prime\circ h)v\right)]\!]=0.
\end{split}
\end{equation*}
The velocity field in the new coordinates reads
$$ u\circ\theta = - k\Big((\pi\circ\theta) \big(\nabla_x v -\frac{\chi\nabla_x h}{1+\chi^\prime h} \partial_y v\big),
\frac{\pi\circ\theta}{1+\chi^\prime h} \partial_y v +(\pi^\prime\circ\theta) v+\gamma\varrho((\pi\circ\theta)v)\Big),$$
and the gradient transforms according to 
$$ \nabla \widehat{=} \left[\begin{array}{c}
\nabla_x  -\frac{\chi\nabla_x h}{1+\chi^\prime h} \partial_y\\
\vspace{-2mm}\\
 \frac{1}{1+\chi^\prime h} \partial_y\end{array}\right].$$

Therefore, balance of mass transforms into the following equation for $v$.
\begin{align*}
& \varrho^\prime((\pi\circ\theta)v) (\pi\circ\theta) (\partial_t v-\frac{\chi\partial_t h}{1+\chi^\prime h} \partial_y v) 
= (\varrho^\prime(p)\partial_t p)\circ\theta
= -{\rm div}(\varrho u)\circ\theta \\
&= \Big(\nabla_x -\frac{\chi \nabla_x h}{1+\chi^\prime h} \partial_y\Big)
\cdot \Big( k\varrho((\pi\circ\theta)v) (\pi\circ\theta) (\nabla_x v - \frac{\chi\nabla_x h}{1+\chi^\prime h} \partial_y v)\Big)\\
&+ \frac{1}{1+\chi^\prime h} \partial_y\Big( k\varrho((\pi\circ\theta)v)\big( \frac{\pi\circ\theta}{1+\chi^\prime h} \partial_y v + (\pi^\prime\circ\theta) v +\gamma\varrho((\pi\circ\theta)v)\Big),
\end{align*}
for all $x\in G$, $y\in (-\underline{h},\overline{h})$, $y\neq 0$, and $t>0$.
Introducing some abbreviations to show the underlying structure, the transformed problem for $(v,h)$ reads in case {\bf (i)}
\begin{equation}
\label{tr-diffusion}
\begin{aligned}
m(v,h) \partial_t v +\cA(v,h) v &= \cF(v,h) &&\mbox{in }\; \Omega\setminus\Sigma,\\
\pi\partial_y v + \pi^\prime v +\gamma \varrho(v\pi)&=0 &&\mbox{on }\; G\times\{-\underline{h},\overline{h}\},\\
\partial_{\nu_G} v&=0  &&\mbox{on }\; \partial G \times \{(-\underline{h},\overline{h})\setminus\{0\}\},
\end{aligned}
\end{equation}
\begin{equation}
\begin{aligned}\label{tr-LY}
[\![(\pi\circ h)v]\!]-\sigma {\rm div}_x(\beta(h)\nabla_x h)&=0 &&\mbox{on }\; \Sigma,\\
\partial_{\nu_G} h &=0  &&\mbox{on }\; \partial G,
\end{aligned}
\end{equation}
and
\begin{equation}
\begin{aligned}\label{tr-ici}
[\![\cB(h)v]\!]&=[\![\cG(v,h)]\!] && \mbox{on }\; \Sigma,\\
\partial_t h + \cB(h)v&=\cG(v,h)&& \mbox{on }\; \Sigma.
\end{aligned}
\end{equation}
In the case {\bf (ii)}, the last equations have to be replaced by
\begin{equation}
\begin{aligned}\label{tr-icii}
[\![\varphi(\varrho((\pi\circ h)v))]\!]&=0 && \mbox{on } \Sigma,\\
[\![\varrho((\pi\circ h)v)]\!]\partial_t h + [\![\varrho((\pi\circ h)v)\cB(h)v]\!]&=[\![\varrho((\pi\circ h)v)\cG(v,h)]\!] && \mbox{on } \Sigma.
\end{aligned}
\end{equation}
Here we employed the following notation, recalling that $\theta=\theta_h$ depends on $h$  and $(\pi\circ\theta)(t,x,0)=\pi(h(t,x))$.
\begin{align}\label{Functions}
m(v,h)&= \varrho^\prime((\pi\circ\theta)v)(\pi\circ\theta), \nn\\
-\cA(v,h)w &= \Big(\nabla_x -\frac{\chi \nabla_x h}{1+\chi^\prime h} \partial_y\Big)\cdot \Big( k\varrho((\pi\circ\theta)v) (\pi\circ\theta) (\nabla_x w - \frac{\chi\nabla_x h}{1+\chi^\prime h} \partial_y w)\Big) \nn \\
&\quad + \frac{1}{1+\chi^\prime h} \partial_y \Big( k\varrho(v\pi\circ\theta)\big( \frac{\pi\circ\theta}{1+\chi^\prime h} \partial_y w\Big), \nn\\
\cF(v,h)&= \frac{1}{1+\chi^\prime h}\Big[\partial_y \big(k\varrho((\pi\circ\theta)v)\big( (\pi^\prime\circ\theta) v  
+\gamma\varrho((\pi\circ\theta)v)\Big)  \\
&\quad + \varrho^\prime((\pi\circ \theta)v) (\pi\circ \theta) \chi \partial_t h \partial_y v\big], \nn \\
\cB(h)w&=  k (\pi\circ \theta) \big((1+|\nabla_x h|^2)\partial_y w -\nabla_x h\cdot\nabla_x w\big) , \nn\\
 \cG(v,h)&=- k\big((\pi^\prime\circ \theta)v +\gamma \varrho((\pi\circ \theta)v)  \big). \nn
\end{align}
These problems consist of a quasilinear parabolic equation for $v$ in the interior, supplemented by an evolution equation for $h$ and two nonlinear transmission conditions on the interface $\Sigma=G\times\{0\}$, besides (nonlinear) boundary conditions at the outer boundary $\partial\Omega$. Of course, initial conditions
$$ v(0)=v_0 \;\; \mbox{on }\Omega\setminus\Sigma,\quad h(0)=h_0 \; \;\mbox{in } G,$$
have to be imposed as well. These are non-standard problems for which no appropriate theory seems to exist. However, we may use the methods introduced by the authors for the Stefan problem with surface tension; in particular, we refer to Pr\"uss and Simonett \cite{PrSi16}, 
Sections 6--11.
In fact, Section 6.7 on the linearized Verigin problem in \cite{PrSi16} is dealing with the case of no boundary contact of the interface and no phase transition. The main point here is to extend this theory to allow for orthogonal boundary contacts, i.e.\ imposing the angle condition
$\nu_G\cdot\nu_\Gamma=0$ at $\partial G\cap \Gamma$. Once we have proved maximal regularity of the principal linearization for this problem,  we then may follow the lines in \cite{PrSi16}, Section 9, to obtain local well-posedness and to construct local state manifolds of the problems.
\subsection{Maximal $L_p$-Regularity}
We remind that in this section $\pi\equiv 1$.
The principal linearization of the two problems then reads as follows 

\begin{equation}
\begin{aligned}\label{plin1}
\partial_t v+\omega v + \cA_0 v &=f_v&&\quad \mbox{in }\;\;\Omega\setminus\Sigma,\\
\partial_y v&=g_b&&\quad \mbox{on }\;\; G\times\{-\underline{h},\overline{h}\},\\
\partial_\nu v &=a_b&&\quad \mbox{on }\;\; \partial G\times\{(-\underline{h},\overline{h})\setminus\{0\}\},
\end{aligned}
\end{equation}
\begin{equation}
\begin{aligned}\label{plin2}
[\![v]\!] +\sigma_0 \cA_\Sigma h &= g_h&&\quad \mbox{in }\;\; \Sigma,\\
\partial_\nu h&= a_h&&\quad \mbox{on }\;\; \partial G,
\end{aligned}
\end{equation}
\begin{equation}
\begin{aligned}\label{plini}
[\![\cB_0 v]\!]&= g_v&&\quad \mbox{in }\;\; \Sigma,\\
\partial_t h +\omega h+{\langle\cB_0 v\rangle} &= f_h&&\quad \mbox{in }\;\; \Sigma.
\end{aligned}
\end{equation}
This is for case {\bf (i)}; in case {\bf(ii)} the last equations have to be replaced by
\begin{equation}
\begin{aligned}\label{plinii}
[\![ v/\varrho_0]\!]&= g_v&&\quad \mbox{in }\;\; \Sigma,\\
\partial_t h +\omega h+[\![\varrho_0\cB_0 v]\!]/[\![\varrho_0]\!] &= f_h&&\quad \mbox{in }\;\; \Sigma.
\end{aligned}
\end{equation}
Here $\cA_0:= \cA(v_0,h_0)/m_0$, ${\cA_\Sigma h} = -( {\Delta_xh} -\beta_0^2\nabla_x h_0\cdot\nabla_x^2h \nabla_xh_0)$, $m_0:= m(v_0,h_0)$, $\cB_0:=\cB(h_0)$,{ $\sigma_0:=\sigma\beta_0$}, $\beta_0:=\beta(h_0)$, 
${\varrho_0=\varrho(v_0)}$, $\nu=\nu_G$, and the brackets {$\langle\cdot\rangle$} indicate the algebraic mean of a quantity across the interface $\Sigma= G\times\{0\}$. {The reason for introducing this mean value is that in \eqref{plini} the function $\mathcal{B}_0v$ has a nontrivial jump across the interface $\Sigma$}.
Furthermore, the functions  $(f_v,f_h,g_b,g_h,g_v,a_b,a_h)$ denote generic inhomogeneities, and of course we have to add initial conditions
$$ v(0)=v_0\quad \mbox{in }\Omega\setminus\Sigma,\quad h(0)=h_0\quad \mbox{in } G.$$
These problems are very similar to the linear Verigin problem studied in \cite{PrSi16}, Section 6.7, however, here we are concerned with a non-smooth domain $\Omega$ and the interface has boundary contact. Therefore, we cannot directly deduce maximal regularity from \cite{PrSi16}, Section 6.7, and we have to add some arguments. 

To ensure that the operators $\cA_\Sigma$ and $\cB_0$ are well-defined in $W_p^{2-1/p}(\Sigma)$ and $W_p^{1-1/p}(\Sigma)$, respectively, we assume in the sequel that $p>d$ and $1/2+1/p\le\mu\le 1$. Under these assumptions, the spaces $W_p^{m-1/p}(\Sigma)$, $m\in\{1,2\}$ are Banach algebras and $W_p^{1+2\mu-3/p}(\Sigma)\hookrightarrow W_p^{2-1/p}(\Sigma)$.

We begin with the definition of the regularity classes of the solutions and the data. In the $L_{p,\mu}$-setting with time weight $t^{1-\mu}$, $1/p+1/2\le \mu\le 1$,
$p\in (d,\infty)$, the solution spaces are according to \cite{PrSi16}, Section 6.7,
\begin{align}\label{solspace}
v&\in H^1_{p,\mu}(\RR_+;L_p(\Omega))\cap L_{p,\mu}(\RR_+;H^2_p(\Omega\setminus\Sigma)),\\
h&\in W^{3/2-1/2p}_{p,\mu}(\RR_+;L_p(G))\cap W^{1-1/2p}_{p,\mu}(\RR_+;H^2_p(G))\cap L_{p,\mu}(\RR_+;W^{4-1/p}_p(G)).\nn  
\end{align}
Trace theory yields the following regularity classes for the data, called {\bf (D)} below.
\goodbreak
\begin{align*}
&{\bf (a)}\; &&v_0\in W^{2\mu-2/p}_p(\Omega\setminus\Sigma),\quad h_0\in W^{2+2\mu-3/p}_p(G);\\
&{\bf (b)}\; &&f_v \in L_{p,\mu}(\RR_+;L_p(\Omega)),\quad f_h \in W^{1/2-1/2p}_{p,\mu}(\RR_+;L_p(G))\cap L_{p,\mu}(\RR_+;W^{1-1/p}_p(G));\\
&{\bf (c)}\; &&g_b(\cdot,k)\in  W^{1/2-1/2p}_{p,\mu}(\RR_+;L_p(G))\cap L_{p,\mu}(\RR_+;W^{1-1/p}_p(G)),\quad k\in\{-\underline{h},\overline{h}\};\\
&{\bf (d)}\;  &&a_b\in  W^{1/2-1/2p}_{p,\mu}(\RR_+;L_p(\partial G\times(-\underline{h},\overline{h}))\cap L_{p,\mu}(\RR_+;W^{1-1/p}_p(\partial G\times\{(-\underline{h},\overline{h})\setminus\{0\}\});\\
&{\bf (e)}\; &&g_h \in W^{1-1/2p}_{p,\mu}(\RR_+;L_p(G))\cap L_{p,\mu}(\RR_+;W^{2-1/p}_p(G));\\
&{\bf (f)}\; &&a_h\in W^{5/4-3/4p}_{p,\mu}(\RR_+;L_p(\partial G))\cap W^{1-1/2p}_{p,\mu}(\RR_+; W^{1-1/p}_p(\partial G))\cap L_{p,\mu}(\RR_+; W_p^{3-2/p}(\partial G));\\
&{\bf (gi)}\; && g_v \in W^{1/2-1/2p}_{p,\mu}(\RR_+;L_p(G))\cap L_{p,\mu}(\RR_+;W^{1-1/p}_p(G)).
\end{align*}
This is for case {\bf (i)}; for case {\bf (ii)} the last condition ought to be replaced by
$$ {\bf (gii)}\quad  g_v \in W^{1-1/2p}_{p,\mu}(\RR_+;L_p(G))\cap L_{p,\mu}(\RR_+;W^{2-1/p}_p(G)).$$
There are several compatibility conditions involved, called {\bf (C)} below.
\goodbreak
\begin{align*}
&{\bf (a)}\; &\partial_y v_0 &= g_b(0)\quad \mbox{on } G\times \{-\underline{h},\overline{h}\}\quad \mbox{for }\mu> 1/2+3/2p;\\
&{\bf (b)}\; & \partial_\nu v_0 &= a_b(0)\quad \mbox{on } \partial G \times \{(-\underline{h},\overline{h})\setminus\{0\}\}\quad \mbox{for }\mu> 1/2+3/2p ;\\
&{\bf (c)}\; & [\![v_0]\!] +\sigma_0\cA_\Sigma h_0&= g_h(0)\quad \mbox{in } G\quad \mbox{for }\mu>3/2p;\\
&{\bf (d)}\; & \partial_\nu h_0 &=a_h(0)\quad \mbox{on } \partial G\quad \mbox{for }\mu>  2/p-1/2;\\
&{\bf (ei)}\; & [\![\cB_0 v_0]\!]&=g_v(0)\quad \mbox{on } G\quad \mbox{for }\mu> 1/2+3/2p;\\
&{\bf (fi)}\; & h_1
&\in W^{4\mu-2-6/p}_p(G)\quad \mbox{for }\mu>  1/2+3/2p;  \\
&{\bf (g)}\;  & [\![a_b(0)]\!] +\partial_\nu\sigma_0\cA_\Sigma h_0 &= \partial_\nu g_h(0)\quad \mbox{for }\mu>  1/2+2/p;\\
&{\bf (h)}\; & \partial_t a_h(0)&= \partial_\nu h_1\quad \mbox{for }\mu> 3/4+7/4p.
\end{align*}
This pertains to case {\bf (i)}, with $h_1:= {f_h}(0)- {\langle\cB_0v_0\rangle}-{\omega h_0}$.
For case \textbf{(ii)} we have to replace {\bf (ei)} and {\bf (fi)} by
\begin{align*}
&{\bf (eii)}\; & &[\![v_0/\varrho_0]\!]=g_v(0)\quad \mbox{on } G\quad \mbox{for }\mu> 3/2p;\\
&{\bf (fii)}\; & &h_1:= {f_h}(0)- [\![\varrho_0\cB_0v_0]\!]/[\![\varrho_0]\!]-{\omega h_0}\in W^{4\mu-2-6/p}_p(G)\quad \mbox{for }\mu>  1/2+3/2p.\\
\end{align*}
Here the first {five} compatibilities are natural, taking the trace of the corresponding boundary conditions. The remaining compatibilities are somewhat hidden: the condition {\bf (f)} comes from the time trace of $\partial_th$ at $t=0$, and the last two conditions follow by taking the normal resp.\ time derivative  of \eqref{plin2} at $\partial G$. Actually, {\bf (g)} is the time trace at time 0 of the compatibility condition
\begin{equation}\label{spcomp}
 [\![a_b]\!] +\partial_\nu\sigma_0\cA_\Sigma h = \partial_\nu g_h.
\end{equation} 
The goal of maximal regularity theory is to prove the converse, i.e.\ if the data are subject to {\bf Conditions (D) and (C)}, then there is a unique solution 
of the problem with regularity \eqref{solspace}.

To achieve this goal, we first reduce to data which are all trivial except for $f_h$. This will be done in several steps.

\medskip

\noindent
{\bf Step 1.} We may reduce to {$(h_0,h_1)=(0,0)$} in the following way. Extend these initial data in their regularity class to all of $\RR^{d-1}$, and set
$$ \bar{h}_1= (2e^{-B_0 t}-e^{-2B_0 t})h_0 + (e^{-B_0^2t}-e^{-2B_0^2t})B_0^{-2}h_1,$$ 
where $B_0=I-\Delta_x$. It is not difficult to see that $\bar{h}_1$ has the right regularity; cf.\ \cite{PrSi16}, Section 6.6.
{Here we note that $(\bar{h}_1(0),\partial_t \bar{h}_1(0))=(h_0, h_1)$.}

\medskip

\noindent
{\bf Step 2.} Next we remove $a_h$ in the following way. We solve the problem
$$
\begin{aligned}
\partial_t \bar{h}_2+\bar{h}_2-\Delta_x \bar{h}_2 &=0 && \mbox{in }\; G,\\
\partial_\nu\bar{h}_2&=a_h-\partial_\nu\bar{h}_1 && \mbox{on }\;\partial G,\\
\bar{h}_2(0) &=0 && \mbox{in }\;\; G. 
\end{aligned}
$$
As $$a_h-\partial_\nu\bar{h}_1\in {_0W}^{1-1/p}_{p,\mu}(\RR_+;H^1_p(\partial G))\cap L_{p,\mu}(\RR_+; W^{3-2/p}_p(\partial G)),$$
the solution $\bar{h}_2$ satisfies
$$\bar{h}_2 \in {_0W}^{3/2-1/2p}_{p,\mu}(\RR_+;H^1_p(G))\cap L_{p,\mu}(\RR_+;W^{4-1/p}_p(G)),$$
i.e.\ has enough regularity, { as the latter space embeds into ${_0 W}^{1-1/2p}_{p,\mu}(\RR_+; H_p^2(G))$,
see for instance formula (4.36) in Section 4.5.5 of \cite{PrSi16} (with $A=I-\Delta,\ \alpha=1$ and $r=1/2$).}
Observe that here we need $\partial G\in C^{4}$. 

We sketch very briefly how to deal with this parabolic problem. Consider the special case of a half space $G=\RR_+^{d-1}$ and split the variable $x=(x',y)$, where $x'\in\RR^{d-2}$ and $y>0$. In this case, one has an explicit solution formula for $h$, which reads
$$h(y)=e^{-Ly}L^{-1}g,$$
where $g:=(a_h-\partial_y h_1)|_{y=0}$ and $L:=(\partial_t+I-\Delta_{x'})^{1/2}$ is the generator of the exponentially stable analytic $C_0$-smigroup $\{e^{-Ly}\}_{y\ge 0}$ in $L_{p,\mu}(\RR_+;L_p(\RR^{d-2}))$. Next, we note that $L^{-1}$ maps 
$${_0W}^{1-1/p}_{p,\mu}(\RR_+;H^1_p(\RR^{d-2}))\cap L_{p,\mu}(\RR_+; W^{3-2/p}_p(\RR^{d-2}))$$
into
$${_0W}^{3/2-1/p}_{p,\mu}(\RR_+;H^1_p(\RR^{d-2}))\cap L_{p,\mu}(\RR_+; W^{4-2/p}_p(\RR^{d-2})),$$
(see e.g. \cite{MS12}). We are now in a position to apply elementary semigroup theory in the space $L_{p,\mu}(\RR_+;H_p^1(\RR^{d-2}))$ to obtain the desired solution class for $\bar{h}_2$. For details, we refer the reader to \cite[Proposition 3.4.3]{PrSi16}. The general case for a bounded domain $G$ with boundary $\partial G\in C^{4}$ follows by means of local coordinates.

\medskip

\noindent
{\bf Step 3.} To remove the compatibility \eqref{spcomp}, we solve the problem
$$
\begin{aligned}
\partial_t \bar{h}_3 +\bar{h}_3-\Delta_x \bar{h}_3 &=0 && \mbox{in }\;G,\\
\partial_\nu\bar{h}_3&= g_3:=-[\![a_b]\!]+\partial_\nu g_h-\partial_\nu\sigma_0\cA_\Sigma(\bar{h}_1+\bar{h}_2) && \mbox{on }\;\partial G,\\
\bar{h}_3(0) &=0,&& \mbox{in }\; G. 
\end{aligned}
$$
As 
$$g_3\in {_0W}^{1/2-1/p}_{p,\mu}(\RR_+;L_p(\partial G))\cap L_{p,\mu}(\RR_+; W^{1-2/p}_p(\partial G)),$$
the unique solution $\bar{h}_3$ belongs to
$$\bar{h}_3 \in {_0W}^{1-1/2p}_{p,\mu}(\RR_+;L_p(G))\cap L_{p,\mu}(\RR_+;W^{2-1/p}_p(G)).$$
{This can be seen as in Step 2.} Then setting $A_\Sigma=\sigma_0\cA_\Sigma$ equipped with Neumann boundary conditions on $\partial G$, the function $\bar{h}_4:=(\partial_t^{1/2}+I+A_\Sigma)^{-1} \bar{h}_3$ belongs to
$$\bar{h}_4 \in {_0W}^{3/2-1/2p}_{p,\mu}(\RR_+;L_p(G))\cap {_0W}^{1-1/2p}_{p,\mu}(\RR_+;H^2_p(G))\cap L_{p,\mu}(\RR_+;W^{4-1/p}_p(G)),$$
i.e.\ it has the right regularity, $\partial_\nu \bar{h}_4=0$, and with $\bar{h}=\bar{h}_1+\bar{h}_2+\bar{h}_4$ we have
\begin{equation}\label{spcomp-bar}
 [\![a_b]\!] +\partial_\nu\sigma_0\cA_\Sigma \bar{h} = \partial_\nu g_h.
\end{equation} 
{\bf Step 4.} Next we remove $v_0$ and $a_b$ in the following way. We restrict these data to the upper part of $\Omega$ resp.\ $\partial \Omega$, i.e.\ $y>0$,
and extend them in the appropriate function space to the cylinder $\Omega\times\RR$ resp.\ $\partial\Omega\times\RR$. Solving the parabolic problem
$$
\begin{aligned}
\partial_t w_2 +w_2-\Delta_x w_2 -\partial_y^2 w_2&=0&& \mbox{in }\; G\times\RR,\\
\partial_\nu w_2&= a_b && \mbox{on }\; \partial G\times\RR,\\
w_2(0)&=v_0&& \mbox{in }\; G\times\RR,
\end{aligned}
$$
the solution $w_2$ belongs to
$$w_2\in H^1_{p,\mu}(\RR_+;L_p(G\times\RR))\cap L_{p,\mu}(\RR_+;H^2_p(G\times\RR)).$$
We do the same thing in the lower part of $\Omega$, to obtain $w_1$ in the same class. Then we set $\bar{v}_1= w_2$ for $y\in(0,\overline{h})$ and
$\bar{v}_1= w_1$ for $y\in (-\underline{h},0)$. This function has the right regularity and trivializes $v_0$ as well as $a_b$.

\medskip

\noindent
{\bf Step 5.} To remove the remaining data {$(f_v,g_b,g_v, g_h)$ we consider the parabolic transmission problem
$$
\begin{aligned}
\partial_t \bar{v}_2 +\omega \bar{v}_2+\cA_0\bar{v}_2 &= f_v -(\partial_t+\omega+\cA_0)\bar{v}_1&& \mbox{in } \;
\Omega\setminus\Sigma,\\
\partial_y \bar{v}_2 &= g_b - \partial_y \bar{v}_1&& \mbox{on }\; G\times\{-\underline{h},\overline{h}\},\\
\partial_\nu \bar{v}_2&=0&& \mbox{on }\; \partial G\times\{(-\underline{h},\overline{h})\setminus\{0\}\},\\
[\![\bar{v}_2]\!]&= g_h -\sigma_0\cA_\Sigma\bar{h}-[\![\bar{v}_1]\!]&& \mbox{in }\; \Sigma,\\
[\![\cB_0 \bar{v}_2]\!]&=g_v-[\![\cB_0 \bar{v}_1]\!] && \mbox{in }\; \Sigma,\\
\bar{v}_2(0)&=0 && \mbox{in }\; \Omega\setminus\Sigma.
\end{aligned}
$$
This problem can be solved in the same way as the transmission problems on smooth domains in \cite{PrSi16}, Section 6.3, employing in addition the reflection principle, to remove the ``corner points" $(x,y)$ where $x\in\partial G$ and  
$y\in\{-\underline{h},0,\overline{h}\}$.

This way we obtain maximal $L_p$-regularity
for this parabolic transmission problem. For an application of the reflection principle, it is important to observe (to be outlined below)
$$\partial_\nu[\![\bar{v}_2]\!]=\partial_\nu(g_h -\sigma_0\cA_\Sigma\bar{h}-[\![\bar{v}_1]\!])=0,$$
on $\partial\Sigma$ by \eqref{spcomp-bar} and the definition of $\bar{v}_1$ which yields $\partial_\nu[\![\bar{v}_1]\!]=[\![a_b]\!]$.

We provide here a brief guideline. Using a localization argument and change of coordinates, one may reduce  the preceeding problem for $\bar{v}_2$ to model problems which are of the following type:
\begin{enumerate}
\item Full space problems;
\item Half space problems;
\item Two-phase problems in a full space;
\item Quarter space problems and
\item Two-phase probems in a half space (with ninety degree contact angle).
\end{enumerate}
The first three types of model problems are well understood in an $L_p$-setting (see e.g.\ \cite[Section 6.7]{PrSi16}). For the last two classes of model problems one may apply the reflection principle to reduce (4) to (2) and (5) to (3). The change of coordinates during the localization procedure {ought} to be carried out in such a way that the normal direction $\nu$ is preserved under the corresponding diffeomorphisms. This way, one may extend all functions by even reflection from a quarter space to a half space or from a two-phase half space to a two-phase problem in a full space, thanks to the condition $\partial_\nu[\![\bar{v}_2]\!]=0$ on $\partial \Sigma$. For a detailed application of the reflection method, we refer the reader to \cite{Wi13}.

\medskip

\noindent
{\bf Step 6.} {By Steps 1 to 5, we may reduce the data in \eqref{plin1}, \eqref{plin2}, \eqref{plini}, and \eqref{plin1}, \eqref{plin2}, \eqref{plinii} to the special case $(f_v,g_b,g_h,g_v,a_b,a_h)=0$ with some (modified) function 
$$\tilde{f}_h\in W^{1/2-1/2p}_{p,\mu}(\RR_+;L_p(G))\cap L_{p,\mu}(\RR_+;W^{1-1/p}_p(G))$$ 
having the property $\tilde{f}_h(0)=0$, provided the trace exists.}

To solve the corresponding problems, one may employ the method explained in \cite{PrSi16}, Section 6.7 in case {\bf (i)} and in \cite{ PrSi16b} for the case {\bf (ii)} as well as the reflection principle, see Step 5. As a consequence, we obtain maximal $L_p$-regularity of the principal linearizations.

\begin{thm}\label{MR-PL}
Let $G\subset\RR^{d-1}$ be a bounded domain with boundary of class $C^{4}$, $\Sigma:=G\times\{0\}$, $\omega>0$, $p\in (d,\infty)$, $1/p+1/2\le\mu\le 1$, and $[\![\varrho_0]\!]\neq0$ in case {\bf (ii)}. \\
Then problems  \eqref{plin1}, \eqref{plin2}, \eqref{plini}, and \eqref{plin1}, \eqref{plin2}, \eqref{plinii} admit a unique solution $(v,h)$ with regularity
\begin{align}\label{solspace-thm}
v&\in H^1_{p,\mu}(\RR_+;L_p(\Omega))\cap L_{p,\mu}(\RR_+;H^2_p(\Omega\setminus\Sigma),\\
h&\in W^{3/2-1/2p}_{p,\mu}(\RR_+;L_p(G))\cap W^{1-1/2p}_{p,\mu}(\RR_+;H^2_p(G))\cap L_{p,\mu}(\RR_+;W^{4-1/p}_p(G)),\nn  
\end{align}
if and only if the data satisfy the corresponding regularity and compatibility conditions {\bf (D)} and {\bf (C)}. The solution depends continuously on the data in the corresponding spaces. The same result holds for $\omega=0$ for finite time intervals $J=(0,t_0)$.
\end{thm} 
As mentioned before, having this maximal $L_p$-regularity result at disposal, we may now follow the arguments in \cite{PrSi16}, Chapter 9, to obtain local well-posedness of the nonlinear problems under considerations. We do not want to repeat the details of the proof here. However, for further reference we define in the original variables the nonlinear state manifolds $\cSM_i$, in which the local semiflow generated by the problems lives, as follows. 
\begin{align}\label{SMi}
(p,h)\in \cSM_i \ \Leftrightarrow\ (p,h)\in W^{2-2/p}_p(\Omega\setminus\Gamma)\times W^{4-3/p}_p(G), \; \mbox{with {\bf(CSMi)}}. 
\end{align}
\begin{align}\label{SMii}
(p,h)\in \cSM_{ii}\  \Leftrightarrow\ (p,h)\in W^{2-2/p}_p(\Omega\setminus\Gamma)\times W^{4-3/}_p(G), \; \mbox{with {\bf(CSMii)}}. 
\end{align}
Here the condition {\bf (CSMi)} is defined by 
$$\begin{aligned}
\partial_\nu p +\gamma\varrho(p)\nu\cdot {\sf e}_d &=0 &&  \mbox{on }\; \partial\Omega\setminus\partial\Gamma,\\
[\![p]\!]-\sigma{\rm div}_x(\beta(h)\nabla_x h) &=0 &&  \mbox{on }\; \Gamma,\\
\partial_{\nu_G} h &=0 &&  \mbox{on }\; \partial G,\\
[\![ k(\partial_y p -\nabla_x h\cdot\nabla_x p +\gamma\varrho(p))]\!]&=0&&  \mbox{on }\; \Gamma.
\end{aligned}$$
Correspondingly, condition {\bf (CSMii)} reads
$$\begin{aligned}
{ }[\![\varrho(p)]\!]&\neq0 &&  \mbox{on }\; \Gamma,\\
\partial_\nu p +\gamma\varrho(p)\nu\cdot {\sf e}_d &=0 &&  \mbox{on }\; \partial\Omega\setminus\partial\Gamma,\\
[\![p]\!]-\sigma{\rm div}_x(\beta(h)\nabla_x h) &=0 &&  \mbox{on } \;\Gamma,\\
\partial_{\nu_G} h &=0 &&  \mbox{on }\;\partial G,\\
[\![\varphi(\varrho(p))]\!]&=0&&  \mbox{on }\; \Gamma.
\end{aligned}$$
The exponent $p\in (1, \infty)$ is chosen in such a way that the initial space $W^{2\mu-2/p}_p\times W^{2+2\mu-3/p}_p$ embeds into $C^1\times C^2$, which requires
$$ 1\geq\mu>\frac{1}{2}+\frac{d+2}{2p},\quad p>d+2.$$
Then the velocity $u$ is well-defined and continuous in the phases, the curvature $H_\Gamma$ is even in $C^1$, and the normal velocity $V_\Gamma$ is continuous. In particular, the equations on the interface are valid pointwise  

\section{Stability of Flat Equilibria}
We first study the full linearization at a flat equilibrium $(p_*,\Sigma)$. Recall that the function $\pi$ of Subsection 3.1 will now 
be chosen as $\pi=p_*$ and  that $p^\prime_*=-\gamma \varrho\circ p_*$.
{Moreover, by shifting $y$ we may assume that $\Sigma=G\times\{0\}$.}

\subsection{Linearization}
In case {\bf(i)} the full linearization at a flat equilibrium $(p_*,\Sigma)$ reads
\begin{equation}
\begin{aligned}\label{lin1}
m_* \partial_t w +\cA_* w &= f_w && \mbox{in }\; \Omega\setminus\Sigma, \\
\cB_* w &=g_b  && \mbox{on }\; G\times \{-\underline{h},\overline{h}\},\\
\partial_{\nu_G} w&=0 && \mbox{on }\; \partial G \times \{(-\underline{h},\overline{h})\setminus\{0\}\}, 
\end{aligned}
\end{equation}

\smallskip
\begin{equation}
\begin{aligned}\label{lin2}
[\![w]\!] +\cA_\Sigma h &=g_h  &&  \mbox{in }\;G,\\
\partial_{\nu_G} h &=0  && \mbox{on }\; \partial G, 
\end{aligned}
\end{equation}

\smallskip
\begin{equation}
\begin{aligned}\label{lin3i}
[\![\cB_* w]\!] &= g_w     &&\mbox{on }\; \Sigma,\\
\partial_t h + \cB_* w &= f_h  && \mbox{on }\; \Sigma, 
\end{aligned}
\end{equation}
where the initial conditions $w(0)=w_0$ and $h(0)=h_0$ have to be added. Here we employed the notations
\begin{equation}\label{AA}
\begin{aligned}
\varrho_*&= \varrho(p_*),\quad \varrho^\prime_*= \varrho^\prime(p_*), \quad m_*= p_*\varrho^\prime_*, 
\quad p_b=p_*(0),\\
\cA_* w &= - {\rm div}_x(kp_*\varrho_*\nabla_x w) -\partial_y(\varrho_*\cB_* w),\\
 \cB_* w &={kp_*(\partial_yw +\gamma(\varrho^\prime_* - \varrho_*/p_*)w),}\\
\cA_\Sigma h &= -(\sigma/p_b)\Delta_x h- (\gamma[\![\varrho_*]\!]/p_b)h. 
\end{aligned}
\end{equation}

\bigskip\noindent
Recall that $[\![p_*]\!]=0$ as the equilibrium interface is flat and therefore, $p_b$ is well-defined.

In the case {\bf(ii)} with phase transition the last two equations have to be replaced by
 \begin{equation}
 \begin{aligned}\label{lin3ii}
[\![w/\varrho_*]\!] &= g_w &&\mbox{on }\; \Sigma,\\
[\![\varrho_*]\!]\partial_t h+ [\![\varrho_*\cB_* w]\!] &= f_h \ && \mbox{on }\; \Sigma,
\end{aligned}
\end{equation}
The problems \eqref{lin1}--\eqref{lin3i} and \eqref{lin1}--\eqref{lin2},\eqref{lin3ii} are lower order perturbations of the principal linearizations studied in the previous section. Therefore, they have the property of maximal $L_p$-regularity as well, in the same function spaces. As a consequence, the underlying operators $L_i$ and $L_{ii}$ defined by
$$ L_i(w,h) = (\cA_* w/m_*,{(\cB_*w)|_\Sigma})\quad L_{ii}(w,h) = (\cA_*w/m_*,[\![\varrho_*\cB_*w]\!]/[\![\varrho_*]\!] )$$
with domains
$${\sf D}(L_i) =\{ (w,h)\in H^2_p(\Omega\setminus \Sigma)\times  W^{4-1/p}_p(G):\, {\bf (BCi)} \},$$
resp.
$${\sf D}(L_{ii}) =\{ (w,h)\in H^2_p(\Omega\setminus \Sigma)\times  W^{4-1/p}_p(G):\, {\bf (BCii)} \},$$
where the boundary and interface conditions are defined by
\begin{equation*}
{\bf (BCi)}\quad
\left.
\begin{aligned}
\cB_*w&= 0 \mbox{ on }  G\times\{-\underline{h},\overline{h}\},
\quad &\partial_{\nu_G} w&=0 \mbox{ on } \partial G\times\{(-\underline{h}, \overline{h})\setminus\{0\}\}\\
{[\![w]\!] +\cA_\Sigma h }&=0  \mbox{ in }  G,
\quad &\partial_{\nu_G} h&=0 \mbox{ on } \partial G, \\
[\![\cB_* w]\!]&=0 \mbox{ on } \Sigma,\quad & {(\cB_* w)|_\Sigma} &\in  {W^{2-2/p}_p(\Sigma),} \\
\end{aligned}
\right.
\end{equation*}
and
\begin{equation*}
{\bf (BCii)}\quad
\left.
\begin{aligned}
\cB_*w&= 0 \mbox{ on }  G\times\{-\underline{h},\overline{h}\},\quad& \partial_{\nu_G} w&=0 \mbox{ on } \partial G\times\{(-\underline{h}, \overline{h})\setminus\{0\}\}\\
{[\![w]\!] +\cA_\Sigma h }&=0  \mbox{ in }  G,
\quad &\partial_{\nu_G} h&=0 \mbox{ on } \partial G, \\
 [\![ w/\varrho_*]\!]&=0 \mbox{ on } \Sigma, \quad &[\![\varrho_* \cB_* &w]\!] \in {W^{2-2/p}_p(\Sigma),} &&
\end{aligned}
\right.
\end{equation*}
are negative generators of analytic $C_0$-semigroups in $X_0= L_p(\Omega)\times W^{2-2/p}_p(G)$ which by boundedness of $\Omega$ are in addition compact. Therefore, the resolvents of these operators are compact as well, and their spectra consist only of countably many eigenvalues of finite multiplicity, which by elliptic regularity are independent of $p$. So it is enough to study these eigenvalues for $p=2$, and to characterize normal stability and normal hyperbolicity of the full linearizations. The main result of this section reads as follows.

\begin{thm}\label{linearstab}
Let $(p_*,\Sigma)$ be a non-degenerate flat equilibrium, $\mu_*:= \gamma[\![\varrho_*]\!]/\sigma$,  and $[\![\varrho_*]\!]\neq0$ in case {\bf (ii)}.
Moreover, let $L_i$ and $L_{ii}$ be the linearizations defined above, and let $\mu_1>0$ denote the smallest nontrivial eigenvalue of the negative Neumann-Laplacian $-\Delta_N$ on $G$.  Then \vspace{2mm}\\
{\bf(a)} $L_i$ and $L_{ii}$ have only real eigenvalues. \vspace{1mm}\\
{\bf (b)} If $\mu_*\leq\mu_1$ then $-L_i$ has no positive eigenvalues, and $-L_{ii}$ has no positive eigenvalues if in addition 
$$  [\![\varrho_*]\!](\varrho_*(-\underline{h})- \varrho_*(\overline{h}))\leq0. $$
{\bf (c)} Suppose $\mu_*\not\in\sigma(-\Delta_N)$. Then ${\rm dim}\, {\sf N}(L_i)=2$ and ${\rm dim}\, {\sf N}(L_{ii})=1$. \vspace{1mm} \\
{\bf (d)}  {Suppose $\mu_*\not\in\sigma(-\Delta_N)$.} 
Then the eigenvalue 0 is semi-simple for $L_i$ and it is semi-simple also for $L_{ii}$ if in addition $ \varrho_*(-\underline{h})\neq \varrho_*(\overline{h})$. \vspace{1mm}\\
{\bf (e)} Assuming conservation of masses, we have $0\not\in \sigma(L_i)$, and also $0\not\in \sigma(L_{ii})$ if in addition $ \varrho_*(-\underline{h})\neq \varrho_*(\overline{h})$.\vspace{2mm}\\
In particular, if $\mu_*<\mu_1$ then $L_i$ is {\bf normally stable}, and $L_{ii}$ is so if in addition 
\begin{equation}\label{normstabii}[\![\varrho_*]\!](\varrho_*(\overline{h})-\varrho_*(-\underline{h}))>0.\end{equation} 
Let $\mu_*\not\in \sigma(-\Delta_N)$. If $\mu_*>\mu_1$ then $L_i$ and $L_{ii}$ are {\bf normally hyperbolic}, and $L_{ii}$ has this property also if $\mu_*<\mu_1$ and the inequality in \eqref{normstabii} is reversed.
\end{thm}
\noindent
Observe that $[\![\varrho_*]\!]\leq 0$ implies $\varrho_*(-\underline{h}) > \varrho_*(\overline{h})$, as $\varrho$ is strictly increasing and $p_*$ is strictly decreasing.
Theorem \ref{linearstab} will be proved in the remainder of this section.

\bigskip

\subsection{The Eigenvalue Problem}
Suppose that $\lambda\in \CC$ is an eigenvalue of {$-L_i$ or $-L_{ii}$} with nontrivial eigenvector $(w,h)$. Multiplying the eigenvalue equation for $w$ with
$\bar{w}p_*/\varrho_*$ and integrating by parts we get, employing the boundary conditions at $\partial \Omega$
\begin{align*}
0&= \lambda \int_\Omega |w|^2 n_* \,d(x,y) +\int_\Omega [k p_*^2|\nabla_x w|^2 + |\cB_* w|^2/k]\,dx\\
&+ \int_G p_*[\![\cB_* w \bar{w}]\!] \,dx.
\end{align*}
with $n_* = m_* p_*/\varrho_*=  \varrho_*^\prime p_*^2/\varrho_*$.
In case {\bf(i)} we obtain by the interface conditions
$$ [\![\cB_* w \bar{w}]\!]= \cB_* w [\![\bar{w}]\!] =\lambda h\overline{\cA_\Sigma h},$$
which implies the fundamental eigenvalue identity
\begin{align}\label{evseq}
0&= \lambda[ \int_\Omega |w|^2 n_* \,d(x,y) + \int_G h p_*\overline{\cA_\Sigma h} \,dx] 
+\int_\Omega [k p_*^2|\nabla_x w|^2 + |\cB_* w|^2/k] \,dx,
\end{align}
In case {\bf(ii)} we obtain accordingly
\begin{align*}
 [\![\cB_* w \bar{w}]\!]&= [\![\varrho_*\cB_*w \bar{w}/\varrho_*]\!] =[\![\varrho_*\cB_*w]\!] \bar{w}/\varrho_*\\
 &=-\lambda h [\![\varrho_*]\!] \bar{w}/\varrho_* = -\lambda h [\![\bar{w}]\!]= \lambda h\overline{\cA_\Sigma h},
\end{align*}
hence we again arrive at identity \eqref{evseq}. Note that
$$\int_G h p_*\overline{\cA_\Sigma h} \,dx = \int_G [\sigma|\nabla_x h|^2 -\gamma[\![\varrho_*]\!]|h|^2] \,dx$$
is real. Identity \eqref{evseq} is of the form $\lambda a + b=0$, where $a,b$ are real and $b\geq0$. If $b\neq0$ this implies $a\neq0$ and $\lambda\neq 0$
is real. If $b=0$ then $\nabla_xw=\cB_*w=0$, which yields $\lambda(h,w)=(0,0)$, hence $\lambda=0$, as we assumed $(w,h)$ to be nontrivial. Therefore, the eigenvalues of $L_i$ and $L_{ii}$ are real.  Furthermore, if $a\geq0$, then $\lambda $ must be negative, which implies linear stability.

So let us investigate when this happens. Clearly it is true if $[\![\varrho_*]\!]\leq0$, so assume {$[\![\varrho_*]\!]>0$.} We then use the decomposition 
$h=h_0 + {\sf h}$, where the mean of $h_0$ equals zero, i.e.
$${\sf h} = |G|^{-1} \int_G h \,dx$$
is the mean of $h$. The identity
$$ \int_G h p_*\overline{\cA_\Sigma h} \,dx = \sigma(-\Delta_N h_0|h_0)_{L_2(G)} - \gamma[\![\varrho_*]\!]( |h_0|_{L_2(G)}^2+|G|{\sf h}^2)$$
shows that this term is nonnegative if ${\sf h}=0$ and $\sigma \mu_1 \geq\gamma[\![\varrho_*]\!]$, where $\mu_1>0$ means the first nontrivial eigenvalue of the negative Neumann-Laplacian $-\Delta_N$ on $G$.  If ${\sf h}\neq0$ then we must take into account the remaining term in the definition of $a$ involving $w$. We may rewrite $a$ in the following form
\begin{align}\label{evseq1}
a&=[ \int_\Omega |w|^2 n_* \,d(x,y) \!- \gamma[\![\varrho_*]\!]|G|{\sf h}^2] + [-\sigma(\Delta_N h_0|h_0)_{L_2(G)} \!- \gamma[\![\varrho_*]\!] |h_0|_{L_2(G)}^2].
\end{align}
We observe that integrating the eigenvalue equation for $w$ and using the divergence theorem we get for $\lambda\neq 0$
$$ \int_\Omega w\, m_* \,d(x,y) = [\![\varrho_*]\!] \int_G h \,dx = [\![\varrho_*]\!]|G| {\sf h}.$$
On the other hand, the Cauchy-Schwarz inequality yields
$$(\int_\Omega  w\,m_* \,d(x,y))^2= (\int_\Omega \frac{m_*}{\sqrt{n_*}} \sqrt{n_*}w \,d(x,y))^2 
\leq \int_\Omega \frac{m_*^2}{n_*} \,d(x,y) \int_\Omega |w|^2 n_* \,d(x,y),$$
with equality if and only if $\sqrt{n_*}w= \alpha m_* /\sqrt{n_*}$ for some $\alpha\in\CC$, which means $w= \alpha\varrho_*/p_*$, i.e.\ $w\in {\sf N}(L_{ii})$.
Thus, the first term in \eqref{evseq1} will be  nonnegative for all $w$ if and only if
$$\gamma c =\gamma \int_{-\underline{h}}^{\overline{h}} \varrho^\prime_*\varrho_* \,dy 
\leq [\![\varrho_*]\!],\quad \mbox{ or equivalently }
\varrho_*(-\underline{h})\leq \varrho_*(\overline{h}) , $$
by \eqref{lcunii}.
If these two conditions hold, then $-L_{ii}$ does not admit positive eigenvalues.

In case {\bf (i)} we proceed in a similar way. Integrating
the eigenvalue equation over the domains $\Omega_j$, we get for $\lambda\neq0$
$$ \int_{\Omega_1} m_* w \,d(x,y) = -|G| {\sf h} \varrho_*(0-),\quad \int_{\Omega_2} m_* w \,d(x,y) = |G| {\sf h} \varrho_*(0+),$$
which by the Cauchy-Schwarz inequality yields
$$ \int_\Omega n_*|w|^2 \,d(x,y) \geq |G| {\sf h}^2 \Big( \frac{\varrho^2_*(0-)}{c_1} + \frac{\varrho^2_*(0+)}{c_2}\Big ),$$
hence by \eqref{lcuni} the first term in \eqref{evseq1} is always nonnegative in case {\bf (i)}.

\bigskip

\subsection{Eigenvalue 0}\label{subsec:eigzero}
For $\lambda=0$ the eigenvalue identity \eqref{evseq} yields
$$\nabla_x w=0,\quad \partial_y w+ \gamma(\varrho^\prime_*-\varrho_*/p_*) w=0\quad \mbox{in } \Omega\setminus\Sigma,$$
hence $w=w(y)$ and as $\varrho_*/p_*$ satisfies these equations we obtain $w(y) = \alpha \varrho_*(y)/p_*(y)$, where $\alpha$ is constant in the phases.
So there are two free parameters for $\alpha$. On the other hand, on $\Sigma$
\begin{equation*}
\begin{aligned}
\sigma \Delta_x h +\gamma[\![\varrho_*]\!] h & = p_*[\![w]\!] = [\![\alpha\varrho_*]\!]  && \mbox{in }\; G,\\
\partial_{\nu_G} h &=0 && \mbox{on }\; \partial G.
\end{aligned}
\end{equation*}
Decomposing $h=h_0 +{\sf h}$ with $\int_G h_0 \,dx=0$ as before, we obtain by taking the mean
$$\gamma[\![\varrho_*]\!] {\sf h} = [\![\alpha \varrho_*]\!],$$
which determines ${\sf h}$ uniquely in the non-degenerate case $[\![\varrho_*]\!]\neq 0$. On the other hand, $h_0$ must satisfy
$$ -\Delta_N h_0 = (\gamma[\![\varrho_*]\!]/\sigma) h_0,$$
hence the number $\mu_*:= \gamma[\![\varrho_*]\!]/\sigma$ is an eigenvalue of the negative Neumann-Laplacian $-\Delta_N$ on $G$, if $h_0$ is nontrivial.

This shows that besides the exceptional case $\mu_*\in\sigma(-\Delta_N)$, the kernel of $L_i$ is two-dimensional. If phase transition is present, we have the additional constraint $[\![w/\varrho_*]\!]=0$, which yields $\alpha$ constant for all $y$, thereby reducing the dimension
of the kernel of $L_{ii}$ to one. In particular, in both cases the dimension of the kernel of the linearization equals the dimension of the tangent space of the manifold of equilibria at the given flat equilibrium $(p_*,\Sigma)$.

Next we  prove that $0$ is not an eigenvalue of $L_\l$, if $\mu_*\not\in \sigma(-\Delta_N)$ and conservation of mass is taken into account. In fact, it is not difficult to show that for the mass functionals $M_j$ we have
$$ \langle M^\prime_2(p_*,0)|(w,h)\rangle = \int_G \int_0^{\overline{h}} m_*w \,dydx -\varrho_*(0+) |G|{\sf h},$$
and
$$ \langle M^\prime_1(p_*,0)|(w,h)\rangle = \int_G \int_{-\underline{h}}^0 m_*w \,dydx +\varrho_*(0-) |G|{\sf h},$$
hence for $M=M_1+M_2$
$$  \langle M^\prime(p_*,0)|(w,h)\rangle = \int_\Omega  m_*w \,d(x,y) -[\![\varrho_*]\!] |G|{\sf h}.$$
So in case {\bf (ii)} we obtain from conservation of total mass
$$\int_\Omega  m_*w \,d(x,y) -[\![\varrho_*]\!] |G|{\sf h}=  \langle M^\prime(p_*,0)|(w,h)\rangle=0$$
as a side condition. Inserting the generic element of the kernel of $L_{ii}$, i.e.\ $ w = \alpha \varrho_*/p_*$ and $\alpha = \gamma{\sf h}$, this constraint yields
$$0= \alpha\int_\Omega m_*\varrho_*/p_* \,d(x,y)-  [\![\varrho_*]\!] |G|{\sf h}= (\alpha |G|/\gamma)( \gamma c -   [\![\varrho_*]\!]),$$
which implies $\alpha=0$, i.e.\ ${\sf N}(L_{ii})=0$, provided  $\gamma c \neq   [\![\varrho_*]\!]$.

On the other hand, in case {\bf (i)} we obtain with conservation of the masses of the phases in a similar way
$$ \alpha_1 c_1 +\varrho_*(0-) {\sf h} =  \alpha_2 c_2 -\varrho_*(0+) {\sf h} = 0,$$
hence by
$$\gamma[\![\varrho_*]\!]{\sf h} = [\![\alpha \varrho_*]\!]= \left(\frac{\varrho^2_*(0-)}{c_1}+\frac{\varrho^2_*(0+)}{c_2} \right){\sf h},$$
which implies ${\sf h}=0$ by \eqref{lcuni}, and then also $\alpha_1=\alpha_2=0$. This proves ${\sf N}(L_i)=0$ in the presence of conservation of masses.

\bigskip

\subsection{Semi-Simplicity of the Eigenvalue 0}
Suppose $(w,h)\in {\sf N}(L_\l)$ and $(w,h)=L_\l(w_1,h_\l)$ for $\l\in\{i,ii\}$. This means in both cases $\cA_* w_1=m_*w$, and $ \cB_*w_1=h$ in case {\bf (i)}, resp.\
$[\![\varrho_*\cB_* w_1]\!] =[\![\varrho_*]\!] h$ in case {\bf(ii)}.
 Multiplying the equation for $w_1$ by $wp_*/\varrho_*$, {integrating by parts  (and using the relation $w(y) p_*(y)/\varrho_*(y)\equiv\alpha$, see Subsection \ref{subsec:eigzero}),} 
 this yields by the boundary conditions on $\partial\Omega$ and the interface conditions on $\Sigma$
\begin{align*}
\int_\Omega |w|^2n_* \,d(x,y)&= \int_\Omega (\cA_* w_1) wp_*/\varrho_* \,d(x,y)\\
&{= \int_G p_*[\![\cB_*w_1w]\!] \,dx= -\int_G p_*\cA_\Sigma h h \,dx 
= \gamma[\![\varrho_*]\!]|G| {\sf h}^2.}
\end{align*}
{Beim letzten Gleichheitszeichen geht $h_0=0$ ein?  Somit sollte in (d)  $\mu_*\notin \sigma(-\Delta_N)$ vorausgesetzt werden.}\\
Inserting $w=\alpha\varrho_*/p_*$ in case {\bf(ii)} this yields by the relation $\alpha= \gamma{\sf h}$ from the previous subsection
$$ \alpha^2 c|G| = \gamma[\![\varrho_*]\!]|G| {\sf h}^2= ([\![\varrho_*]\!]/\gamma)|G|\alpha^2.$$
Hence $\alpha=0$ and so $(w,h)=(0,0)$ in case $[\![\varrho_*]\!]\neq \gamma c$.
On the other hand, in case {\bf (i)}, integrating the equation for $w_1$ over $\Omega_j$ yields
\begin{align*}
\alpha_1c_1= -\rho_*(0-) {\sf h},\quad \alpha_2 c_2 = \rho_*(0_+) {\sf h},
\end{align*}
and so in this case we get
$$ \gamma[\![\varrho_*]\!]|G| {\sf h}^2=\int_\Omega |w|^2n_* \,d(x,y)= |G|(\alpha_1^2c_1 + \alpha_2^2 \alpha_2)
=|G|{\sf h}^2 \left(\frac{\varrho^2_*(0-)}{c_1}+\frac{\varrho^2_*(0+)}{c_2} \right),$$
and so $({\sf h},\alpha,w)=(0,0,0)$ by \eqref{lcuni}. This shows that $0$ is always semi-simple in case {\bf(i)}, and if in addition
$\varrho_*(-\underline{h})\neq\varrho_*(\overline{h})$ in case {\bf (ii)}, thanks to \eqref{lcunii}.
 Recall that we always assume that $\varrho_*$ is non-degenerate, i.e.\ $[\![\varrho_*]\!]\neq0$, if phase transition is admitted.

\bigskip

\subsection{Positive Eigenvalues}
We shall investigate the existence of positive eigenvalues of {$-L_\l$,} $\l\in\{i,ii\}$. 
This will be achieved by reduction to an eigenvalue problem for $h$ in $L_2(G)$ by a suitable Neumann-to-Dirichlet operator. 
We will need several steps.\vspace{3mm}\\
{\bf (a)} First we introduce operators $A_i$ and $A_{ii}$ in $L_2(\Omega\setminus\Sigma)$ as follows.
$$ A_\l w := \cA_*w/m_*,\quad w\in {\sf D}(A_\l),\; \l\in\{i,ii\},  $$
$$ {\sf D}(A_\l)= \{w\in H^2_2(\Omega\setminus\Sigma): \, {\bf (bck)}\; \mbox{holds}\},\quad \l\in\{i,ii\},$$
where
\begin{equation*}
{\bf (bci)}\qquad
\left.
\begin{aligned}
\cB_*w&= 0 \mbox{ on }  G\times\{-\underline{h},\overline{h}\}, \quad& \partial_{\nu_G} w&=0 \mbox{ on } \partial G\times\{(-\underline{h}, \overline{h})\setminus\{0\}\},\\
 {\cB_* w}&=0 \mbox{ on } \Sigma,  \quad &
\end{aligned}
\right.
\end{equation*}
and
\begin{equation*}
{\bf (bcii)}\quad
\left.
\begin{aligned}
\cB_*w&= 0 \mbox{ on }  G\times\{-\underline{h},\overline{h}\},\quad &\partial_{\nu_G} w&=0 \mbox{ on } \partial G\times\{(-\underline{h}, \overline{h})\setminus\{0\}\},\\
 [\![ w/\varrho_*]\!]&=0 \mbox{ on } \Sigma,\quad &[\![\varrho_*\cB_* w]\!] &=0 \mbox{ on } \Sigma.
\end{aligned}
\right.
\end{equation*}
As an inner product in $Y:=L_2(\Omega)$ we employ
$$ (w_1|w_2)_Y:= \int_\Omega w_1\overline{w}_2 n_* \,dx,$$
which is equivalent to the usual one, as $n_*>0$ is bounded from above and away from $0$.
It is not difficult to prove that $A_\l$, $\l\in\{i,ii\}$, are self-adjoint w.r.t.\ this inner product and positive semi-definite. 
In fact, by an integration by parts we have
$$(A_\l w|w)_Y = \int_G[ kp_*^2|\nabla_x w|^2 + |\cB_* w|^2/k] \,dx\geq0.$$
As $G$ is bounded and has $C^2$-boundary, these operators have compact resolvents, and hence their spectra consist only of nonnegative eigenvalues of finite algebraic multiplicity which are also semi-simple. In particular, we have the orthogonal decomposition
$ Y= {\sf N}(A_\l)\oplus {\sf R}(A_\l)$, with associated orthogonal projections $P_\l$ onto ${\sf N}(A_\l)$. The kernels of these operators are easily identified:
$${\sf N}(A_i)={\rm span}\{ \chi_j\varrho_*/p_*:j=1,2\},\quad {\sf N}(A_{ii})={\rm span}\{\varrho_*/p_*\},$$
with $\chi_j$ the characteristic functions of $[-\underline{h},0]$ for $j=1$ and of $[0,\overline{h}]$ for $j=2$. Then the projections $P_\l$ are given by
$$ P_iw = \frac{1}{|G|c_1}(\chi_1\varrho_*/p_*|w)_Y \, \chi_1\varrho_*/p_* +  \frac{1}{|G|c_2}(\chi_2\varrho_*/p_*|w)_Y\,\chi_2\varrho_*/p_*,$$
and
$$ P_{ii} w= \frac{1}{|G| c}(\varrho_*/p_*|w)_Y \, \varrho_*/p_*.$$

\noindent
{\bf (b)} Next, for $\lambda>0$ we define the relevant Neumann-to-Dirichlet operators by means of $T_\lambda g :=[\![w]\!]$, where $w$ solves the problem
\begin{align}\label{wlambda}
m_*\lambda w +\cA_* w&=0 \quad \mbox{in } \Omega\setminus\Sigma,\nn\\
\cB_* w &=0 \quad \mbox{on } G\times \{-\underline{h},\overline{h}\},\\
\partial_{\nu_G}w&=0 \quad \mbox{on }\partial G \times \{(-\underline{h}, \overline{h})\setminus\{0\}\},\nn
\end{align}
supplemented by
$$ [\![\cB_* w]\!]=0,\quad -\cB_* w =g \quad \mbox{in } G$$
in case {\bf (i)} and by
$$ [\![w/\varrho_*]\!]=0,\quad -[\![ \rho_* \cB_*w]\!]=[\![\varrho_*]\!]g \quad \mbox{in } G$$
in case {\bf (ii)}. \\
Then with $B_\lambda =\lambda T_\lambda + \cA_\Sigma$, with ${\sf D}(B_\lambda)={\sf D}(\Delta_N)$,
$\lambda>0$ is an eigenvalue of {$-L_\l$} if and only if $B_\lambda h=0$ for some $h\neq0$.  This reduces the eigenvalue problem for the operators $L_\l$ to a problem for the selfadjoint operators $B_\lambda$ on $Z=L_2(G)$.

It is not difficult to show that $T_\lambda$ is in both cases selfadjoint in $Z$, and by an integration by parts it is easy to verify the following crucial identity, showing in particular that $T_\lambda$ is positive semi-definite, for each $\lambda>0$.
\begin{equation}\label{Tlambda}
(T_\lambda g|g)_Z = {\frac{1}{p_b} }\left[ \lambda \int_\Omega |w|^2 n_* \,d(x,y) + \int_\Omega [kp_*^2|\nabla_x w|^2 + |\cB_*w|^2/k] \,d(x,y)\right].
\end{equation}
This identity shows that in both cases $T_\lambda$ is injective for $\lambda>0$, and that its inverse is positive definite in $Y$.

Below we will show that $B_\lambda$ is positive definite for large enough $\lambda$, and that it admits negative eigenvalues for small $\lambda$, if the the instability conditions $\mu_*>\mu_1$, or in case {\bf (ii)} in addition
$\varrho_*(-\underline{h})<\varrho_*(\overline{h})$, are valid. Then as $\lambda$ varies from $0$ to $\infty$ at least one eigenvalue of $B_\lambda$ must cross the imaginary axis through zero, thereby inducing at least one positive eigenvalue of
{$-L_\l$} as claimed. We can even determine  the {\em Morse-index} $m_\l$ of $L_\l$, i.e.\ the dimension of the unstable eigenspace of $L_\l$:
 we have
$$m_1 = \sum_{\mu_l<\mu_*} {\rm dim}\,{\sf N}(\Delta_N+\mu_l),$$
and $m_2=m_1$ if   $\varrho_*(-\underline{h})>\varrho_*(\overline{h})$, and $m_2=m_1+1$ if $\varrho_*(-\underline{h})<\varrho_*(\overline{h})$.
Here we assume, as before, $\mu_*\not\in \sigma(-\Delta_N)$ as well as  $[\![\varrho_*]\!]\neq 0$ and $\varrho_*(-\underline{h})\neq\varrho_*(\overline{h})$ for the case {\bf (ii)}.

\bigskip

\noindent
{\bf (c)} {\em Small $\lambda$}\\
We assume w.l.o.g.\ $[\![\varrho_*]\!]>0$. In this step we determine
$$ B_0 := \lim_{\lambda\to0+} B_\lambda = \lim_{\lambda \to 0+} \lambda T_\lambda + \cA_\Sigma$$
in the strong sense. Let $w_\lambda$ denote the solution of \eqref{wlambda} with $g=h$ and $w_1$ that of \eqref{wlambda} with $\lambda=\lambda_1$, 
{$\lambda_1>0$ a fixed number,}
and $g=h$.
Then $w_0=w_\lambda-w_1$ has homogeneous boundary conditions, i.e\ we have $(\lambda+A_\l)w_0 = (\lambda_1- \lambda) w_1$, which yields
$$ \lambda w_\lambda = \lambda w_0 +\lambda w_1 = \lambda(\lambda +A_\l)^{-1} (\lambda_1-\lambda)w_1 +\lambda w_1\to \lambda_1 P_\l w_1,$$
as $\lambda\to0+$. By {\bf (a)} we have
\begin{align*}
\lambda_1 P_{ii}w_1 &= \frac{\varrho_*}{p_*c|G|}\lambda_1\int_\Omega (\varrho_*/p_*) w_1 n_* \,d(x,y) 
= \frac{\varrho_*}{p_*c|G|}\int_\Omega \lambda_1 w_1 m_* \,d(x,y)\\
&= -\frac{\varrho_*}{p_*c|G|}\int_\Omega \cA_* w_1 \,d(x,y) = -\frac{\varrho_*}{p_*c|G|}\int_G [\![\varrho_*\cB_* w_1]\!] \,dx \\
&= \frac{\varrho_*}{p_*c} [\![\varrho_*]\!] {\sf h},
\end{align*}
where, as before, ${\sf h}$ denotes the mean value of $h$. This implies
$$ \lambda T_\lambda h = \lambda[\![w_\lambda]\!]\to [\![\lambda_1 P_{ii} w_1]\!] = \frac{[\![\varrho_*]\!]^2}{ p_b c}{\sf h},$$
hence with $h_0=h-{\sf h}$
$$ B_0 h = \lim_{\lambda\t0+} B_\lambda h = -\frac{\sigma}{p_b} \Big( \Delta_N +\mu_*\Big)h_0 + 
\frac{1}{p_b c}\Big(([\![\varrho_*]\!]-\gamma c)[\![\varrho_*]\!]\Big) {\sf h}.$$
This shows that $B_0$ has negative eigenvalues if and only if $\mu_1<\mu_*$, or $[\![\varrho_*]\!]<\gamma c$ which by \eqref{lcunii} is equivalent to
$\varrho_*(-\underline{h}) > \varrho_*(\overline{h})$.

On the other hand, in case {\bf (i)} we have similarly
$$\lambda_1 P_i w_1 = \left(-\frac{\varrho_*\varrho_*(0-)\chi_1}{c_1 p_b} +\frac{\varrho_*\varrho_*(0+)\chi_2}{c_2 p_b}\right) {\sf h},$$
and so
$$B_0 h = -\frac{\sigma}{p_b} \Big(\Delta_N + \mu_*\Big) h_0 + 
\frac{1}{p_b}\Big( \frac{\varrho^2_*(0-)}{c_1}+ \frac{\varrho^2_*(0+)}{c_2}-\gamma[\![\varrho_*]\!]\Big){\sf h}.$$
In virtue of \eqref{lcuni}, the second term is nonnegative, hence we see that $B_0$ has negative eigenvalues if and only if $\mu_*>\mu_1$.

\bigskip

\noindent
{\bf (d)} {\em Large $\lambda$}\\
We claim that there are $\lambda_0>0$ and $\eta>0$ such that
\begin{equation}\label{largelambda}
(B_\lambda h|h)_Z \geq \eta |h|_Z^2,\quad {\lambda\ge \lambda_0,}\quad h\in {\sf D}(-\Delta_N).
\end{equation}
To prove this  let us assume the contrary. Then there are sequences $\lambda_n\to\infty$, $h_n\in {\sf D}(-\Delta_N)$, $|h_n|_Z=1$, such that
$$(B_{\lambda_n} h_n|h_n)_Z = \lambda_n(T_{\lambda_n} h_n|h_n)_Z + (\cA_\Sigma h_n|h_n)_Z\leq 1/n.$$
With \eqref{Tlambda} this shows for the solution $w_n=w_{\lambda_n}$ of \eqref{wlambda} for $\lambda=\lambda_n$
$$ \lambda_n |w_n|_{L_2(\Omega)} +\sqrt{\lambda_n} |\nabla w_n|_{L_2(\Omega)} 
+ |h_n|_{H^1_2(G)} \leq C,\quad n\geq 1,$$
hence there is $w_\infty\in {L_2(\Omega)}$ such that $\lambda_n w_n\to w_\infty$  
weakly in ${L_2(\Omega)}$, up to a subsequence.
Then for a test function $\phi\in \cD(\Omega\setminus\Sigma)$ we have
$$\lambda_n \langle w_n|\phi \rangle = -\langle\cA_*w_n|\phi \rangle = -\langle w_n|\cA_*\phi \rangle
=-(1/\lambda_n)\langle \lambda_n w_n|\cA_*\phi \rangle \to 0 $$
as $n\to \infty$, hence $w_\infty=0$. Let $P$ denote the orthogonal projection onto the unstable part of $\cA_\Sigma$, i.e.\ onto the space spanned by the eigenspaces of all nonpositive eigenvalues of $\cA_\Sigma$, and let $Q=I-P$ be its complementary projection.
As $\cA_\Sigma$ is selfadjoint, there is an orthonormal basis $\{a_k\}_{k=1}^N$ of eigenvectors of $A_\Sigma$ spanning ${\sf R}(P)$. We extend each function $a_k$ constantly in $y$ to $\Omega$. Then we obtain in case {\bf (ii)} by an integration by parts
\begin{align*}
[\![\varrho_*]\!] (h_n|a_k)_Z &= (-[\![\varrho_* \cB_* w_n]\!]|a_k)_Z
= \int_{\Omega\setminus\Sigma} {\rm div}(\varrho_*(kp_*\nabla_x w_n + \cB_*w_n{\sf e}_d)a_k) \,d(x,y)\\
&= -\langle \cA_*w_n |a_k\rangle + \int_\Omega k\varrho_*p_*\nabla_x w_n \cdot\nabla_x a_k \,d(x,y)\\
&= \langle\lambda_n w_n| a_k\rangle  + \int_\Omega k\varrho_*p_*\nabla_x w_n \cdot\nabla_x a_k \,d(x,y)\to 0
\end{align*}
as $n\to\infty$ for each $k=1,\ldots,N$, since $\lambda_n w_n\to0$ weakly and $\sqrt{\lambda_n}\nabla w_n$ is bounded  in $L_2(\Omega)$.
With $[\![\varrho_*]\!]\neq0$ this shows $Ph_n\to0$ as $n\to\infty$, and so also $Qh_n\to 0$ as $\cA_\Sigma$ is positive definite on ${\sf R}(Q)$. This leads to a contradiction to $|h_n|_Z=1$, thereby proving the claim.

In case {\bf (i)} the arguments are similar, hence we omit details here. Summarizing, we have proved all assertions of  Theorem 4.1.

\bigskip

\noindent
\subsection{Nonlinear Stability of Flat Equilibria}

\noindent
Let {$\cE$} be the set of (non-degenerate) flat equilibria, and fix some equilibrium {$e_*=(\pi_*,\Gamma_*)\in\cE$}.
Employing the findings from the previous section, we have

\begin{itemize}
\item {$e_*$} is {\bf normally stable} if  $\mu_*<\mu_1$ in case {\bf (i)}, and in case {\bf (ii)} if additionally 
$[\![\varrho_*]\!](\varrho_*(\overline{h})-\varrho_*(-\underline{h}))>0$.
\vspace{1mm}
\item {$e_*$} is {\bf normally hyperbolic} if $\mu_1<\mu_*\not\in \sigma(-\Delta_N)$ in both cases, and in  case {\bf (ii)} also if
$\mu_*<\mu_1$ and $[\![\varrho_*]\!](\varrho_*(\overline{h})-\varrho_*(-\underline{h}))<0$.
\end{itemize}

\noindent
Therefore, the {generalized principle of linearized stability}  due to Pr\"uss, Simonett, Zacher \cite{PSZ09}  yields  our
main result on stability of equilibria, i.e.\ on the {\em Rayleigh-Taylor instability} for the Verigin problem in a finite capillary.


\goodbreak
\noindent
\begin{thm} Let $p>d+2$ and {$e_*\in\cE$} be a {non-degenerate} flat equilibrium such that 
$$\mu_*:=  \gamma[\![\varrho_*]\!]/\sigma \not\in \sigma(-\Delta_N)$$ and in case {\bf (ii)} in addition
$[\![\varrho_*]\!](\varrho_*(\overline{h})-\varrho_*(-\underline{h}))\neq0$.
Then
\begin{itemize}
\item[{\bf (i)}] If {$e_*$}  is normally stable, it is nonlinearly stable, and any solution starting near {$e_*$}
is global and converges to another equilibrium {$e_\infty\in\cE$} exponentially fast.
\vspace{1mm}
\item[{\bf (ii)}]
If {$e_*$} is normally hyperbolic, then {$e_*$} is nonlinearly unstable. Any solution starting in a neighborhood of {$e_*$}
{and staying near} {$e_*$} exists globally and converges to an equilibrium {$e_\infty\in\cE$} exponentially fast.
\end{itemize}
\end{thm}
\begin{proof}
The proof parallels that for the Stefan problem with surface tension given in  \cite[Chapter 11]{PrSi16},
see also  Pr\"uss, Simonett and Zacher \cite{PSZ13}.

In the following, we shall outline  the strategy of the proof, without providing all the technical details.
We first note that the nonlinear problems \eqref{p1} and \eqref{p2} are equivalent to
\begin{equation}\label{nonlinear-i}
\begin{aligned}
m_* \partial_t  w +\cA_*  w &=F_w(w,h) && \mbox{in }\; \Omega\setminus\Sigma, \\
\cB_*  w &=G_b(w,h)  && \mbox{on }\; G\times \{-\underline{h},{h}\},\\
\partial_{\nu_G}  w&=0 && \mbox{on }\; \partial G \times \{(-\underline{h},{h})\setminus\{0\}\}, \\
[\![ w]\!] +\cA_\Sigma  h &=G_h(w,h)  && \mbox{in }\; G,\\
\partial_{\nu_G}  h &=0  && \mbox{on }\; \partial G, \\
[\![\cB_* w]\!] &= [\![G^i_w (w,h) ]\!]    &&\mbox{on }\; \Sigma,\\
\partial_t h + \cB_*  w &= F^i_h(w,h)  && \mbox{on }\; \Sigma ,
\end{aligned}
\end{equation}
and
\begin{equation}\label{nonlinear-ii}
\begin{aligned}
m_* \partial_t  w +\cA_*  w &=F_w(w,h) && \mbox{in }\; \Omega\setminus\Sigma, \\
\cB_*  w &=G_b(w,h)  && \mbox{on }\; G\times \{-\underline{h},{h}\},\\
\partial_{\nu_G}  w&=0 && \mbox{on }\; \partial G \times \{(-\underline{h},{h})\setminus\{0\}\}, \\
[\![ w]\!] +\cA_\Sigma  h &=G_h(w,h)  && \mbox{in }\; G,\\
\partial_{\nu_G}  h &=0  && \mbox{on }\; \partial G, \\
[\![w/\varrho_*]\!] &= [\![ G^{ii}_w (w,h) ]\!]     &&\mbox{on }\; \Sigma,\\
[\![\varrho_*]\!] \partial_t  h + [\![ \varrho_* \cB_*  w ]\!] &= F^{ii}_h(w,h)  && \mbox{on }\; \Sigma, 
\end{aligned}
\end{equation}
respectively.
Using the notation of \eqref{AA} and  \eqref{Functions} (with $\pi=p_*$, $v=1+w$),
we have for the nonlinear terms
\begin{equation*}
\begin{aligned}
F_w(w,h) &= \cF(1+w,h) -\cA(1+w, h)w +\cA_* w + (m_* - m(1+w, h )) \partial _t w, \\
G_b(w,h) &= k [ (p^\prime_* + \gamma   p_* \varrho^\prime_*) w -  p^\prime_*(1+w)  -\gamma\varrho(p_*(1+w))], \\
G_h(w,h) &= -(1/p_b)\big[  [\![ (p_*\circ \theta_h) (1+w) ]\!]
- \sigma\, {\rm div}\, (\beta(h)\nabla_x h)  \big]  + \cA_\Sigma h, \\
G^i_w(w,h) &=  \cG (1+w,h)-B(h)w  +\cB_*w,    \\
F^i_h (w,h) &= \cG (1+w,h)-B(h)w  +\cB_*w 
\end{aligned}
\end{equation*}
in case {\bf (i)}, and 
\begin{equation*}
\begin{aligned}
G^{ii}_w(w,h) &= \varphi(\varrho((p_*\circ \theta)(1+w))- w/\varrho_*   ,\\
F^{ii}_h (w,h) &=[\![\varrho((p_*\circ\theta)(1+w)) (\cG (1+w,h)-\cB(h)w +\varrho_* \cB_* w )]\!] \\
  & +  [\![ \varrho_*  -\varrho((p_*\circ\theta)(1+w))]\!]\partial_t h 
\end{aligned}
\end{equation*}
in case {\bf (ii)}.
Note that the function $F_w$ contains the terms $\partial_t h$ and $\partial_t w$, 
while $F^{ii}_h$ contains the term $\partial_t h$. For these we may use the substitutions\\
\begin{equation*}
\begin{aligned}
&\partial _t w= \frac{1}{m(1+w,h)} \big( \cF(1+w,h) -\cA (1+w,h)w \big), \\
&\partial_t h=  \cG(1+w,h) - \cB(h) w   \quad\text{in case {\bf (i)}},\\
&\partial_t h= \frac{1}{[\![\varrho((p_*\circ\theta)(1+w)) ]\!]} [\![\varrho((p_*\circ\theta)(1+w)) (\cG (1+w,h)-\cB(h)w)]\!]\quad\text{in case {\bf (ii)}}. 
\end{aligned}
\end{equation*}
One readily verifies that 
\begin{equation}\label{zero-der}
\begin{aligned}
&(F_w(1,0), F^\prime_w(1,0)) =(0,0),                               &&(G_l (1,0), G^\prime_l(1,0)) = (0,0),   \\
 &(G^{j}_w(1,0), (G^{j}_w)^\prime (1,0)) =(0,0),   &&(F^{j}_h (1,0), (F^{j}_h)^\prime (1,0))  = (0,0), 
\end{aligned}
\end{equation}
where $l\in\{b,h\}$ and  ${j}\in \{i,ii\}$.
The state manifolds near the equilibrium $(p_*,\Sigma)$ can then be described by
\begin{equation*}
\cSM_*^{i} =\big\{(w,h)\in W^{2-2/p}_p(\Omega\setminus \Gamma)\times W^{4-3/p}_p(G):  {\bf (N B Ci)}\text{  holds}\big\}
\end{equation*}
in case {\bf (i)}, with {\bf (NBCi)} given by
\begin{equation*}
\left.
\begin{aligned}
\cB_*w&= G_b(w,h) \mbox{ on }  G\times\{-\underline{h},\overline{h}\},
\quad &\partial_{\nu_G} w&=0 \mbox{ on } \partial G\times\{(-\underline{h}, \overline{h})\setminus\{0\}\}\\
 [\![w]\!] +\cA_\Sigma h &=G_h(w,h)  \mbox{ in }  G,
\quad &\partial_{\nu_G} h&=0 \mbox{ on } \partial G, \\
[\![\cB_* w]\!]&=[\![G^i_w(w,h)]\!] \mbox{ on } \Sigma,\quad &\cB_* w &- F^i_h(w,h)\in W^{2-6/p}_p(\Sigma). \\
\end{aligned}
\right.
\end{equation*}
For case {\bf (ii)}, we have
\begin{equation*}
\cSM_*^{ii} =\big\{(w,h)\in W^{2-2/p}_p(\Omega\setminus \Gamma)\times W^{4-3/p}_p(G):  {\bf (NBCii)}\text{ holds}\big\},
\end{equation*}
where {\bf (NBCii)} is now given by
\begin{equation*}
\left.
\begin{aligned}
\cB_*w&= G_b(w,h) \mbox{ on }  G\times\{-\underline{h},\overline{h}\},
\quad &\partial_{\nu_G} w&=0 \mbox{ on } \partial G\times\{(-\underline{h}, \overline{h})\setminus\{0\}\}\\
 [\![w]\!] +\cA_\Sigma h &=G_h(w,h)  \mbox{ in }  G,
\quad &\partial_{\nu_G} h&=0 \mbox{ on } \partial G, \\
[\![w/\varrho_*]\!]&=[\![G^{ii}_w(w,h)]\!] \mbox{ on } \Sigma,
\quad &[\![ \varrho_*\cB_* w ]\!] &- F^{ii}_h(w,h)\in W^{2-6/p}_p(\Sigma). \\
\end{aligned}
\right.
\end{equation*}
In the sequel, we focus on case {\bf (i)}. Let
\begin{equation*}
\begin{aligned}
      \cSX^{i}_*  &  :=\big\{z=(w,h)\in W^{2-2/p}_p(\Omega\setminus \Gamma)\times W^{4-3/p}_p(G):  {\bf (LBCi)}\text{ holds} \big\}, \\
 X^{i}_\gamma & := \big\{z=(w,h)\in W^{2-2/p}_p(\Omega\setminus \Gamma)\times W^{4-3/p}_p(G):  \partial_{\nu_G}w=0,\;  \partial_{\nu_G}h=0  \big\}, \\
Y^i_\gamma    & : =  W^{1-3/p}_p(G\times \{-\underline{h}, \bar h\})\times W^{2-3/p}_p(G)\times W^{1-3/p}_p(\Sigma), \\
{\sf B}^ {i}z        & := (\cB_* w, [\![ w ]\!]+\cA_\Sigma h, [\![\cB_*w ]\!]), \\
{\sf G}^{i}(z)   &  := (G_b(z), G_h(z), [\![G^{i}_w(z)]\!]). 
 \end{aligned}
\end{equation*}
{with {\bf (LBCi)} given by
\begin{equation*}
\left.
\begin{aligned}
\cB_*w&= 0 \mbox{ on }  G\times\{-\underline{h},\overline{h}\},
\quad &\partial_{\nu_G} w&=0 \mbox{ on } \partial G\times\{(-\underline{h}, \overline{h})\setminus\{0\}\},\\
{[\![w]\!] +\cA_\Sigma h }&=0  \mbox{ in }  G,
\quad &\partial_{\nu_G} h&=0 \mbox{ on } \partial G, \\
[\![\cB_* w]\!]&=0 \mbox{ on } \Sigma,\quad & \cB_* w &\in  W^{2-6/p}_p(\Sigma). \\
\end{aligned}
\right.
\end{equation*}
}
Note that ${\sf G}^i\in C^1(X^i_\gamma, Y^i_\gamma)$, $({\sf G}^i(0), ({\sf G}^i)^\prime (0))=(0,0)$,
and that
\begin{equation*} 
\begin{aligned}
\cSX^i_*  &= \big\{ \tilde z=(\tilde w,\tilde h)\in X^i_\gamma : {\sf B}^i \tilde z  = 0,                   
  &&{\cB_* \tilde w \in W^{2-6/p}_p(\Sigma) \big\}, } \\
\cSM^i_* &= \big\{ z=(w,h)\in X^i_\gamma :  {\sf B}^i z = {\sf G}^i(z),   \!\!  &&{\cB_* w - F^i_h(z)\in W^{2-6/p}_p(\Sigma)\big\}.  }\\
\end{aligned}
\end{equation*}
Let $\omega>0$ be sufficiently large. By similar arguments as in Section 3.2 one shows that  the linear elliptic transmission problem
\begin{equation} \label{elliptic-linear}
\begin{aligned}
m_* \omega\; \overline w +\cA_* \overline w &=0 && \mbox{in }\; \Omega\setminus\Sigma, \\
\cB_* \overline w &=g_b  && \mbox{on }\; G\times \{-\underline{h},\overline{h}\},\\
\partial_{\nu_G} \overline w&=0 && \mbox{on }\; \partial G \times \{(-\underline{h},\overline{h})\setminus\{0\}\}, \\
[\![\overline w]\!] +\cA_\Sigma \overline h &=g_h && \mbox{in }\; G,\\
\partial_{\nu_G} \overline h &=0  && \mbox{on }\; \partial G, \\
[\![\cB_* \overline w]\!] &= g_w    &&\mbox{on }\; \Sigma,\\
\omega \overline h + {\langle \cB_* \overline w\rangle } &= f_h  && \mbox{on }\; \Sigma,
\end{aligned}
\end{equation}
admits for each 
$(g_b, g_h,g_w, f_h)\in Y^i_\gamma\times W^{1-3/p}_p(\Sigma)$
 a unique solution $(\overline w,\overline h)\in X^i_\gamma$, 
{provided $p\in (d+2,\infty)$?}
Hence, by  \eqref{zero-der} and the implicit function theorem, there exists  $r>0$ and a mapping $\phi^i \in C^1( B_{\cSX^i_*}(0,r), X^i_\gamma)$ with $(\phi^i(0),(\phi^i)^\prime(0))=(0,0)$
such that $\overline z=(\overline w,\overline h)=\phi^i(\tilde z)$ is, for each $\tilde z\in  B_{\cSX^i_*}(0,r)$, the unique solution of the nonlinear elliptic transmission problem
\begin{equation}\label{elliptic-nonlinear}
\begin{aligned}
m_* \omega\; \overline w +\cA_* \overline w &=0 && \mbox{in }\; \Omega\setminus\Sigma, \\
\cB_* \overline w &=G_b(\tilde z +\overline z)  && \mbox{on }\; G\times \{-\underline{h},\overline{h}\},\\
\partial_{\nu_G} \overline w&=0 && \mbox{on }\; \partial G \times \{(-\underline{h},\overline{h})\setminus\{0\}\}, \\
[\![\overline w]\!] +\cA_\Sigma \overline h &=G_h(\tilde z +\overline z)  && \mbox{in }\; G,\\
\partial_{\nu_G} \overline h &=0  && \mbox{on }\; \partial G, \\
[\![\cB_* \overline{w}]\!] &= [\![G^i_w (\tilde z +\overline z) ]\!]    &&\mbox{on }\; \Sigma,\\
\omega \overline h + {\langle \cB_* \overline w\rangle }&  
= {\langle  F^i_h(\tilde z +\overline z)\rangle}  && \mbox{on }\; \Sigma.
\end{aligned}
\end{equation}
{We note that the last two lines in \eqref{elliptic-nonlinear} and the fact that $F^i_h = G^i_w$
imply 
$$ B_*\overline w- F^i_h(\tilde z +\overline z) =\langle B_*\overline w- F^i_h(\tilde z +\overline z)\rangle 
=-\omega \overline h\in W^{4-3/p}_p(G).$$}
Hence,  we can conclude that
\begin{equation*}
\begin{aligned}
& {\sf B}^i [\tilde z +\phi^i (\tilde z)] = {\sf G}^i(\tilde z +\phi^i (\tilde z)),  \\
& \cB_* [\tilde w + \phi_1^i(\tilde z)] - F^i_h(\tilde z + \phi^i(\tilde z))\in W^{2-6/p}_p(\Sigma)
\end{aligned}
\end{equation*}
for each $\tilde z\in B_{\cSX^i_*}(0,r)$, where $\phi_1^i(\tilde z)$ denotes the first component of $\phi^i(\tilde z)$.
{Moreover, one shows that the mapping ${\rm id} +\phi^i$ is surjective onto a neighborhood of zero in $X^i_\gamma$. This readily implies} that the mapping 
$$\Phi^i:  B_{\cSX^i_*}(0,r)\to \cSM^i_*, \quad \Phi^i(\tilde z):=\tilde z +\phi^i(\tilde z),$$
provides a local parameterization of the state manifold $\cSM^i_*$ over $\cSX^i_*$ near $(0,0)$,
with tangent space $T_0\cSM^i_*$ isomorphic to $\cSX^i_*$.

Using the notation introduced above,  problem \eqref{nonlinear-i} can be rewritten in condensed form as
\begin{equation*} 
\begin{aligned}
\partial_t z +{\sf A}^i z & = {\sf F}^i(z), \\
                    {\sf B}^i z & = {\sf G}^i(z), \\
                               z(0)& =z_0,
\end{aligned}
\end{equation*}
where ${\sf A}^i z:=L_i(w,h)$, ${\sf F}^i(z)= (F_w(w,h)/m_*, F^i_h(w,h))$, and $ z_0=(w_0,h_0),$
with $L_i$  as  in Section 4.1.
%
We can now follow~\cite{PrSi16}, Sections 11.2.4 and 11.3, to establish the
assertions of the theorem. 
\end{proof}


\begin{thebibliography}{99}

\bibitem{BiSo00} G.~I.~Bizhanova, V.~A.~Solonnikov, {\em On problems with free boundaries for second-order parabolic equations}.
Algebra i Analyzi {\bf 12}, 98--139 (2000).
Translation in {St. Petersburg Math.~J.} {\bf 12},  949--981 (2000).

\bibitem{EEM13} M.~Ehrnstr{\"o}m, J.~Escher, and B.-V. Matioc,
{\em Steady-state fingering patterns for a periodic {M}uskat problem.}
Methods Appl. Anal. {\bf 20}, 33--46 (2013).

\bibitem{EMM12a} J.~Escher, A.-V. Matioc, and B.-V. Matioc,
{\em A generalized {R}ayleigh-{T}aylor condition for the {M}uskat problem}.
 Nonlinearity {\bf 25}, 73--92 (2012).

\bibitem{EsMa11} J.~Escher and B.-V. Matioc,
{\em On the parabolicity of the {M}uskat problem: well-posedness, fingering, and stability results}.
Z. Anal. Anwend.  {\bf 30}, 193--218 (2011).

\bibitem{EMW15} J.~Escher, B.-V. Matioc, and C.~Walker,
{\em  The domain of parabolicity for the {M}uskat problem}.
arXiv:1507.02601.

\bibitem{Fro99} E.~V.~Frolova, {\em Estimates in $L_p$ for the solution of a model problem corresponding to the Verigin problem}.
Zap. Nauchn. Sem. {S-Petersburg.} Otdel. Mat. Inst. Steklov {\bf 259}, 280--295 (1999).
Translation in J.~Math. Sci. (N.~Y.)) {\bf 109},  2018--2029 (1999).

\bibitem{Fro03} E.~V.~Frolova, {\em Solvability of the Verigin problem in Sobolev spaces}.
Zap. Nauchn. Sem. {S-Petersburg.} Otdel. Mat. Inst. Steklov {\bf 295}, 180--203 (2003).
Translation in {J.~Math. Sci. (N.~Y.) } {\bf 127}, 1923--1935  (2003).

\bibitem{GuTi10} Y.~Guo, I.~Tice,
{\em Linear Rayleigh-Taylor instability for viscous, compressible fluids}.
SIAM J. Math. Anal. {\bf 42}, 1688--1720  (2010).

\bibitem{JTW16} J.~Jang, I.~Tice,  and Y.~Wang, 
{em The compressible viscous surface-internal wave problem: nonlinear Rayleigh-Taylor instability}.
Arch. Ration. Mech. Anal. {\bf 221}, 215--272 (2016).

\bibitem{Ma17} B.-V. Matioc,
{\em Viscous displacement in porous media: the Muskat problem in 2D}.
arXiv:1701.00992. 


\bibitem{MS12} M.~Meyries, R.~Schnaubelt, \emph{Interpolation, embeddings and traces of anisotropic fractional Sobolev spaces with temporal weights}.  J.~Funct.~Anal. \textbf{262}, 1200--1229s (2012).

\bibitem{PSZ09} J.~Pr{\"u}ss, G.~Simonett, and R.~Zacher, 
{\em On convergence of solutions to equilibria for quasilinear parabolic problems}. 
 J.~Differential Equation \textbf{246}, 3902--3931 (2009).

\bibitem{PrSi10} J.~Pr{\"u}ss, G.~Simonett.
\newblock {\em On the {R}ayleigh-{T}aylor instability for the two-phase  {N}avier-{S}tokes equations.}
\newblock Indiana Univ. Math.~J. {\bf 59}, 1853--1871 (2010).

\bibitem{PrSi16} J.~Pr\"uss, G.~Simonett, {\em Moving Interfaces and Quasilinear Parabolic Evolution Equations}.
Monographs in Mathematics {\bf 105}, Birkh\"auser,  Basel 2016.

\bibitem{PrSi16a} J.~Pr\"uss, G.~Simonett, {\em On the Muskat problem.}
\newblock   Evol. Equ. Control Theory  {\bf 5}, 631--645 (2016).

\bibitem{PrSi16b} J.~Pr\"uss, G.~Simonett, 
\newblock {\em On the Verigin problem with and without phase transition.}
\newblock Interfaces Free Bound. {\bf 20}, 107--128 (2018).
 
\bibitem{PSZ13} J.~Pr{\"u}ss, G.~Simonett, and R.~Zacher,
{\em Qualitative behavior of solutions for thermodynamically consistent Stefan problems with surface tension. }
Arch. Ration. Mech. Anal. {\bf 207},  611--667 (2013).

\bibitem{Rad04} E.~V.~Radkevich, {\em The classical Verigin-Muskat problem, the regularization problem, and inner layers}.
Sovrem. Mat. Prilozh. {\bf 16}, 113--155 (2004).
Translation in J. Math. Sci. (N.Y.), 1000--1044 (2004). 

\bibitem{Tao97} Y.~Tao, {\em Classical solutions of Verigin problem with surface tension}.
Chinese Ann. Math. Ser.~B {\bf 18}, 393--404 (1997).

\bibitem{TaYi96} Y.~Tao, F.~Yi, {\em Classical Verigin problem as a limit case of Verigin problem with surface tension at free boundary}.
Appl. Math. J. Chinese Univ.Ser.B {\bf 11}, 307--322 (1996).


\bibitem{WaTi12} Y.~Wang, I.~Tice,
{\em The viscous surface-internal wave problem: nonlinear Rayleigh-Taylor instability}. 
Comm. Partial Differential Equations  {\bf 37}, 1967--2028  (2012).

\bibitem{Wi13} {M.~Wilke, {\em {R}ayleigh-{T}aylor instability for the two-phase
  {N}avier-{S}tokes equations with surface tension in cylindrical domains}.
\newblock Habilitations-Schrift Universit\"at Halle, Naturwissenschaftliche Fakult\"at II,   2013.
 arXiv:1703.05214. }

\bibitem{Xu97} L.~F.~Xu, {\em A Verigin problem with kinetic condition}.
Appl. Math. Mech. {\bf 18}, 177-184 (1997).

\bibitem{Zh17a} Y. Zhou, 
{\em Rayleigh-Taylor and Richtmyer-Meshkov instability induced flow, turbulence, and mixing. I.}
 Phys. Rep. {\bf 720/722}, 1--136 (2017). 

 \bibitem{Zh17b}Y.  Zhou, 
 {\em Rayleigh-Taylor and Richtmyer-Meshkov instability induced flow, turbulence, and mixing. II.}
 Phys. Rep. {\bf 723/725}, 1--60 (2017).

 
\end{thebibliography}
\end{document}